\documentclass{article}
\usepackage[utf8]{inputenc}
\usepackage[OT1]{fontenc}
\def\affaddr{}
\newcommand{\email}[1]{{\small\texttt{#1}}}

\makeatletter
\def\ps@pprintTitle{%
 \def\@oddfoot{\footnotesize\itshape\hfill\today}%
}%
\makeatother

\usepackage{amsthm}
\usepackage{amssymb}
\usepackage{amsmath}
\usepackage{amsfonts}
\usepackage{bm} % For bold symbols
\usepackage{mathrsfs} %provides \mathscr
\usepackage{multicol}
\usepackage[naturalnames]{hyperref}

\theoremstyle{definition}
\newtheorem{definition}{Definition}[section]

\newtheorem{remark}[definition]{Remark}
\newtheorem{example}[definition]{Example}

\theoremstyle{plain}
\newtheorem{lemma}[definition]{Lemma}
\newtheorem{proposition}[definition]{Proposition}
\newtheorem{theorem}[definition]{Theorem}
\newtheorem{corollary}[definition]{Corollary}

\newcommand{\bxi}{{{\mathbf{\xi}}}}
\newcommand{\bmu}{{{\mathbf{\mu}}}} 
\newcommand{\bnu}{{{\mathbf{\nu}}}} 
\newcommand{\blambda}{{\mbox{\boldmath$\lambda$}}}
\newcommand{\bpartial}{{\mbox{\boldmath$\partial$}}}
\newcommand{\balpha}{{\alpha}}
\newcommand{\bbeta}{{\beta}}
\newcommand{\bgamma}{{\gamma}}
\newcommand{\bLambda}{\mathbf{\Lambda}}
\newcommand{\pp}[1]{d_{#1}} 

\def\nil{o}
\def\rank{{\rm rank}}
\def\mM{{\tt M}}

\def\xb{\mathbf{x}}

\def\fb{\mathbf{f}}
\def\gb{\mathbf{g}}

\def\d{{\bf d}}

\def\bx{{\bf x}}
\def\Cc{{\mathcal C}}
 
\def\Nc{{\mathcal N}} 

\def\bz{{\bf z}}
\def\bM{{\bf M}}
\def\e{{\bf e}}

\def\K{{\mathbb C}}
\def\KK{{\mathbb K}}

\def\I{{\mathcal I}}

\def\QQ{{Q}}

\def\C{{\mathbb C}}

\def\N{{\mathbb N}}
\def\Q{{\mathbb Q}}

\def\ord{\mathrm{ord}}
\def\sp{{\rm{span}}}
\def\DDD{{\mathscr D}}
\def\m{{\mathfrak m}}
\def\mult{{\delta}}
\def\bmu{{\mu}}

\begin{document}

%\conferenceinfo{ISSAC}{'15, Bath, UK}
%\begin{frontmatter}
\title{On deflation and multiplicity structure}
\author{
\hspace{-0.75cm}\begin{tabular}{ccc}
 \begin{minipage}{5.5cm}
Jonathan D. Hauenstein\thanks{Research partly supported by DARPA YFA, NSF grant 
ACI-1460032, and Sloan Research Fellowship.}\\
      \affaddr{Department of Applied and Computational Mathematics and Statistics,}\\
       \affaddr{University of Notre Dame}\\
      \affaddr{Notre Dame, IN, 46556, USA}\\
       \email{hauenstein@nd.edu}
% 2nd. author
\end{minipage}
& 
\begin{minipage}{4.5cm}
Bernard Mourrain\\
\affaddr{Inria Sophia Antipolis% M\'editerran\'ee
}\\
\affaddr{2004 route des Lucioles, B.P. 93,}\\
\affaddr{06902 Sophia Antipolis,}\\
\affaddr{Cedex France}\\
       \email{bernard.mourrain@inria.fr }
% 3rd. author
\end{minipage} 
& 
\begin{minipage}{6cm}
Agnes Szanto\thanks{Research partly supported by NSF grant CCF-1217557.}\\
       \affaddr{Department of Mathematics, }\\
       \affaddr{North~Carolina~State~University}\\
       \affaddr{Campus Box 8205,}\\
       \affaddr{Raleigh, NC, 27965, USA.}\\
       \email{aszanto@ncsu.edu}\\
%\and  % use '\and' if you need 'another row' of author names
\end{minipage}
\end{tabular}
}
\date{ }
\maketitle

\begin{abstract}
  This paper presents two new constructions related to singular solutions of polynomial systems. 
  The first is a new deflation method for  
  an isolated singular root. This construction uses a single linear
  differential form defined from the Jacobian matrix of the input, and
  defines the deflated system by applying this differential form to
  the original system.  The advantages of this new deflation is that it
  does not introduce new variables and the increase in the number of
  equations is linear in each iteration instead of the quadratic increase of previous
  methods. The second construction gives the coefficients of
  the so-called inverse system or dual basis, which defines the
  multiplicity structure at the singular root.  We present a system of
  equations in the original variables plus a relatively small number
  of new variables that completely deflates the root in one step.  We
  show that the isolated simple solutions of this new system 
  correspond to roots of the original system with given multiplicity structure up to a given order.  
  Both constructions are ``exact'' in that they permit one to 
  treat all conjugate roots simultaneously and can be used 
  in certification procedures for singular roots and 
  their multiplicity structure with respect to 
  an exact rational polynomial system.
\end{abstract}

% \begin{keyword} deflation \sep multiplicity structure \sep Newton's method \sep inverse system \sep multiplication matrix

% \end{keyword}
% \end{frontmatter}

\section{Introduction}

One issue when using numerical methods for solving polynomial
systems is the ill-conditioning and possibly erratic behavior 
of Newton's method near singular solutions.  
Regularization (deflation) techniques remove
the singular structure to restore
local quadratic convergence of~Newton's~method.

Our motivation for this work is twofold. On one hand, in a recent
paper \cite{AkHaSz2014}, two of the co-authors of the present paper
and their student studied a certification method for approximate roots
of exact overdetermined and singular polynomial systems, and wanted to
extend the method to certify the multiplicity structure at the root as
well. Since all these problems are ill-posed, in \cite{AkHaSz2014} a
hybrid symbolic-numeric approach was proposed, that included the exact
computation of a square polynomial system that had the original root
with multiplicity one. In certifying singular roots, this exact square
system was obtained from a deflation technique that added
subdeterminants of the Jacobian matrix to the system
iteratively. However, the multiplicity structure is destroyed by this
deflation technique, that is why it remained an open question how to
certify the multiplicity structure of singular roots of exact
polynomial systems.

Our second motivation was to find a method that simultaneously refines
the accuracy of a singular root and the parameters describing the
multiplicity structure at the root. In all previous numerical
approaches that approximate these parameters, they apply numerical
linear algebra to solve a linear system with coefficients depending on
the approximation of the coordinates of the singular root. Thus the
local convergence rate of the parameters was slowed from the quadratic
convergence of Newton's iteration applied to the singular roots. We
were interested if the parameters describing the multiplicity
structure can be simultaneously approximated with the coordinates of
the singular root using Newton's iteration. Techniques that
additionally provide information about the multiplicity structure of a
singular root can be applied to bifurcation analysis of ODEs and PDEs
(see, e.g. \cite{FriedmanHu2006,Haoetal2012}).  They can also be
helpful in computing the topological degree of a polynomial map
\cite{EiLe77} or for analyzing the topology of real algebraic curves (see
e.g. \cite{alberti:inria-00343110} and Example 6.2 in
\cite{mantzaflaris:inria-00556021}).

In the present paper, we first give an improved version of the
deflation method that can be used in the certification algorithm of
\cite{AkHaSz2014}, reducing the number of added equations at each
deflation iteration from quadratic to linear. We prove that applying a
single linear differential form to the input system, corresponding to
a generic kernel element of the Jacobian matrix, already reduces both
the multiplicity and the depth of the singular root. Furthermore, we
study how to use this new deflation technique to compute isosingular
decompositions introduced in~\cite{HauWam13}.

Secondly, we give a description of the multiplicity structure using a
polynomial number of parameters, and express these parameters together
with the coordinates of the singular point as the roots of a
multivariate polynomial system. We prove that this new polynomial
system has a root corresponding to the singular root but now with
multiplicity one, and the newly added coordinates describe the
multiplicity structure. Thus, this second approach completely deflates
the system in one step. The number of equations and variables in the
second construction depends polynomially on the number of variables
and equations of the input system and the multiplicity of the singular
root. Moreover, we also show that the isolated simple solutions of our
extended polynomial system correspond to roots of the original system
that have prescribed multiplicity structure up to a given order.

Both constructions are exact in the sense that approximations of the
coordinates of the singular point are only used to detect numerically
non-singular submatrices, and not in the coefficients of the
constructed polynomial systems.

This paper is an extended version of the ISSAC'15 conference paper \cite{Hauensteinetal2015}.

\subsection{Related work.}

 The treatment of singular roots is a critical issue for numerical analysis 
with a
large literature on methods that transform the problem into a new one
for which Newton-type methods converge quadratically~to~the~root.
 
% DEFLATION WITH NO NEW PARAMETER
Deflation techniques which add new equations in order to
reduce the multiplicity were considered in
\cite{Ojika1987199,Ojika1983463}.
By triangulating the Jacobian matrix at the (approximate) root,
new minors of the polynomial Jacobian matrix are added to the initial
system in order to reduce the multiplicity of the singular solution.

A similar approach is used in \cite{HauWam13} and \cite{GiuYak13}, 
where a maximal invertible block of the Jacobian matrix 
at the (approximate) root is computed and minors of the 
polynomial Jacobian matrix are added to the
initial system. 
For example, when the Jacobian matrix at the root vanishes,
all first derivatives of the input polynomials are added to the system
in both of these approaches.  
Moreover, it is shown in \cite{HauWam13} that deflation can be performed
at nonisolated solutions in which the process stabilizes to 
so-called {\em isosingular sets}. 
At each iteration of this deflation approach, 
the number of added equations can be taken to be $(N-r)\cdot(n-r)$,
where $N$ is the number of input polynomials,
$n$ is number of variables, 
and $r$ is the rank of the Jacobian at the root.

These methods repeatedly use their constructions 
until a system with a simple root is obtained.

In~\cite{Lecerf02}, a triangular presentation of the ideal in a
good position and derivations with respect to the leading variables are used
to iteratively reduce the multiplicity. This process is applied for p-adic
lifting with exact computation.

% DEFLATION WITH NEW PARAMETERS

% WITH PERTURBATION PARAMETERS
In other approaches, new variables and new equations are introduced simultaneously.
For example, in \cite{YAMAMOTO:1984-03-31}, 
new variables are introduced to describe some perturbations 
of the initial equations and some differentials which
vanish at the singular points.
This approach is also used in \cite{LiZhi2014}, where it is shown that
this iterated deflation process yields a system with a simple root.

In \cite{mantzaflaris:inria-00556021}, 
perturbation variables are 
also introduced in relation with the inverse system of the singular point
to obtain directly a deflated system with a simple root.
The perturbation is constructed from a monomial basis of the local
algebra at the multiple root.

% WITH DIFFERENTIAL PARAMETERS
In~\cite{lvz06,lvz08}, only variables for the differentials of the initial
system are introduced.
The analysis of this deflation is improved in \cite{DaytonLiZeng11},
where it is shown that the number of steps is bounded by the order
of the inverse system.
This type of deflation is also used in \cite{LiZhi2013}, for the special case
where the Jacobian matrix at the multiple root has rank $n-1$ (the 
breadth one case).

In these methods, at each step, 
both the number of variables and equations are increased,
but the new equations are linear in the newly added variables. 

% COMUTING THE INVERSE SYSTEM
The aforementioned deflation techniques usually break the structure of the local
ring at the singular point. The first method to compute the inverse system 
describing this structure is due to F.S. Macaulay \cite{mac1916}
and known as the dialytic method.
More recent algorithms for the construction of inverse systems are described
in \cite{Marinari:1995:GDM:220346.220368} which 
reduces the size of the intermediate linear systems (and exploited in
\cite{Stetter:1996:AZC:236869.236919}). 
In \cite{Leykin2008}, an approach related to the dialytic method is used to compute all isolated and embedded components 
of an algebraic set. 
The dialytic method had been further improved in~\cite{Mourrain97} and, 
more recently, in 
\cite{mantzaflaris:inria-00556021},
using an integration method. This technique reduces significantly the
cost of computing the inverse system, since it relies on the solution
of linear systems related to the inverse system truncated in some
degree and not on the number of monomials in this degree.
Singular solutions of polynomial systems have been
studied by analyzing multiplication matrices 
(e.g.,~\mbox{\cite{Corless97,Moller95,SGPhT09}}) via non-local methods, 
which apply to the zero-dimensional case.

The computation of inverse systems has also been used to approximate a 
multiple root. The dialytic method is used in~\cite{zeng05} and the relationship 
between the deflation approach and the inverse system is analyzed, 
exploited, and implemented in \cite{HSZ13}.
In~\cite{Pope2009606}, a minimization approach is used to reduce the value of
the equations and their derivatives at the approximate root, assuming a basis
of the inverse system is known. 
In \cite{GLSY07}, the certification of a multiple root with breadth one is obtained 
using $\alpha$-theorems.
In %\cite{Wu:2008:CMS:1390768.1390812}
\cite{WuZhi2011}, the inverse system is constructed via
Macaulay's method, tables of multiplications are deduced, and their
eigenvalues are used to improve the approximated root. They 
show that the convergence is quadratic at the multiple root. 
In \cite{LiZhi12}, they show that in the breadth one case 
the parameters needed to describe the 
inverse system is small, and use it to compute the singular roots in \cite{LiZhi12b}.
% when the Jacobian has corank 1 at the multiple root.
The inverse system has further been exploited in deflation 
techniques in \cite{mantzaflaris:inria-00556021}.  
This is the closest to our approach as it computes a perturbation of 
the initial polynomial system with a given inverse system, deduced
from an approximation of the singular solution.
The inverse system is used to transform directly the singular root into a simple root of an augmented system.

\subsection{Contributions.} 

In this paper, we present two new constructions. The first one is a
  new deflation method for a system of polynomials with an isolated
  singular root which does not introduce new
parameters. At each step, a single differential of the system is
considered based on the analysis of the Jacobian at the singular
point.   The advantage of this new deflation is that it reduces the number of added equations at each 
deflation iteration from quadratic to linear.  We prove that the resulting deflated system
  has strictly lower multiplicity and depth at the singular point than
  the original one.

 In addition to the results that appeared in \cite{Hauensteinetal2015}, in the present extended version of the paper we 
 study the relationship of the new deflation method to the {\em
   isosingular deflation} (see Proposition \ref{determ}), and show how
 to use our  deflation technique to compute an {\em isosingular
   decomposition}  of an algebraic set,  introduced in
 \cite{HauWam13}  (see Section \ref{isosing}). 
%Dynamic construction of the Newton iteration.
%It may increase the degree.

 Secondly, to approximate efficiently both the singular point and its
 multiplicity structure, we propose a new deflation which involves
 fewer number of new variables compared to other approaches that rely
 on Macaulay's dialytic method. It is based on a new characterization
 of the isolated singular point together with its multiplicity
 structure via inverse systems. The deflated polynomial system
 exploits the nilpotent and commutation properties of the
 multiplication matrices in the local algebra of the singular
 point. We prove that the polynomial system we construct has a root
 corresponding to the singular root but now with multiplicity one, and
 the new added coordinates describe the multiplicity structure.

 This new method differs dramatically from previous deflation methods.
All other deflation methods in the literature use an
iterative approach that may apply as many iterations as the maximal
order  of the derivatives of the input polynomials that vanish at the
root. At each iteration these traditional deflation techniques at
least double the number polynomial equations, and either introduce new
variables, or greatly increase the degrees of the new
polynomials. Thus these deflation techniques grow exponentially in the
number of iterations and are considered very inefficient when more
than 2 iterations are needed. Our new technique 
completely deflates the root in a single iteration,  introducing both  new variables and new polynomials to
the system. The number of new variables and polynomials  are
quadratic in the multiplicity of the point, and the degrees also
remain  bounded by the original degrees and the multiplicity.
More precisely, the number of variables and equations in this construction
is at most $n+n\mult(\mult -1)/2$ and
 $N\mult+n(n-1)(\mult-1)(\mult-2)/4$, respectively, where $N$~is the
 number of input polynomials, $n$~is the number of variables, and
 $\mult$~is the multiplicity of the singular point. The degrees of the
 polynomials in the new system are bounded by the degrees of the input
 system plus the {\em order} of the root, i.e. the maximal order of
 the differentials that vanish at the root. 
Thus, it is the first deflation technique that produces a deflated system which
has polynomial size in the multiplicity and in the size of the input.

In this extended version we also give a new construction, called
{\em $E$-deflated ideals}, which is a modification of  {\em deflated ideals}
introduced in \cite{Leykin2008}.  While the construction in
\cite{Leykin2008} uses Macaulay's dialytic method, our construction is
based on our deflation method using multiplication matrices, which
results in introducing significantly fewer auxiliary variables. We prove that
the isolated simple roots of the $E$-deflated ideal correspond to
roots of the original system that have a prescribed multiplicity
structure up to a given order (see Section \ref{DeflIdeal}).

\section{Preliminaries}

Let $\fb:= (f_1, \ldots, f_N)\in \KK[\bx ]^N$ with  $\bx =(x_1, \ldots, x_n)$ for some $\KK\subset \C$ field. Let $\bxi=(\xi_1, \ldots, \xi_n)\in \C^n$ be an isolated multiple  root of $\fb$. 
Let $I=\langle f_1, \ldots, f_N\rangle$, $\m_{\xi}$ be the maximal ideal at ${\xi}$ and $\QQ$ be the primary component of $I$ at $\bxi$ so that $\sqrt{\QQ}=\m_{\xi}$.\\

Consider the ring of power series $\K[[\bpartial_\xi]]:= \K[[\partial_{1,\xi}, \ldots, \partial_{n,\xi}]]$ and we use the notation 
for  $\bbeta=(\beta_1, \ldots, \beta_n)\in \N^n$:
$$\bpartial^{\bbeta}_{\xi}:=\partial_{1, \xi}^{\beta_1}\cdots \partial_{n, \xi}^{\beta_n}.$$ 
We identify $\C[[\bpartial_\xi]]$ with the dual space $\C[\bx ]^*$  by
considering $\bpartial^{\bbeta}_{\xi}$ as derivations and evaluations at $\bxi$,  defined by 
\begin{equation}\label{partial}
\bpartial^{\bbeta}_{ \xi}(p):=\bpartial^\bbeta(p)\bigg|_{\bxi} := \frac{\partial^{|\bbeta|  } p}{\partial x_1^{\beta_1}\cdots \partial x_n^{\beta_n}} (\bxi) \quad \text{ for } p\in \C[\bx].
\end{equation}
Hereafter, the derivations ``at $\xb$'' will be denoted
$\bpartial^{\bbeta}$ instead of $\bpartial^{\bbeta}_{\xb}$. The
derivation with respect to the variable $\partial_{i}$ in $\K[[\bpartial]]$ is denoted
$\pp{\partial_{i}}$ $(i=1,\ldots, n)$.
Note that 
$$
 \frac{1}{\bbeta!} \bpartial^{\bbeta}_{ \xi}((\bx-\bxi)^\balpha)=\begin{cases} 1  & \text{ if } \balpha=\bbeta, \\
0 & \text{ otherwise,}
\end{cases}
$$
where $\bbeta! = \beta_1!\cdots \beta_n!$.

For $p\in \K[\bx]$ and $ \Lambda\in \K[[\bpartial_\xi]]=\K[\bx]^{*}$,
let 
$$
p\cdot \Lambda: q \mapsto \Lambda (p\,q).
$$
We check that $p=(x_i-\xi_i)$ acts as a derivation on
$\C[[\bpartial_\xi]]$:
$$ 
(x_{i}-\xi_{i}) \cdot \bpartial^{\bbeta}_{ \xi}= \pp{\partial_{i, \xi}} (\bpartial^{\bbeta}_{ \xi})
$$

For an ideal $I\subset \K[\bx]$, let $I^{\perp}=\{\Lambda \in
\K[[\bpartial_{\xi}]]\mid \forall p\in I, \Lambda (p)=0\}$.
The vector space $I^{\perp}$ is naturally identified with the dual
space of $\K[\bx]/I$.
We check that $I^{\perp}$ is a vector subspace of $\K[[\bpartial_{\xi}]]$,
which is stable by the derivations ${d_{\partial_{i, \xi}}}$.

\begin{lemma}\label{lem:primcomp}
If $Q$ is a $\m_{\xi}$-primary isolated component of $I$, then $Q^{\perp}=I^{\perp}\cap\K[\bpartial_{\xi}]$.
\end{lemma}

This lemma shows that to compute $Q^{\perp}$, it suffices to compute
all polynomials of $\K[\bpartial_{\xi}]$ which are in $I^{\perp}$.
Let us denote this set $\DDD= I^{\perp}\cap\K[\bpartial_{\xi}]$. It is a
vector space stable under the derivations
$\pp{\partial_{i, \xi}}$. Its dimension is the dimension of
$Q^{\perp}$ or $\K[\bx]/Q$, that is the {\em multiplicity} of
$\xi$, denote it by $\mult_{\xi} (I)$, or simply by $\mult$  if $\xi$ and $I$ is clear from the context.

For an element $\Lambda(\bpartial_\xi) \in \K[\bpartial_\xi]$ we
define the {\em order}  ${\rm ord}(\Lambda)$ to be the maximal
$|\bbeta|$ such that $\bpartial^{\bbeta}_{ \xi}$ appears in
$\Lambda(\bpartial_\xi)$ with non-zero coefficient.  

For $t\in \N$, let $\DDD_{t}$ be the elements of $\DDD$ of order $\leq t$.
As $\DDD$ is of dimension $d$, there exists a smallest $t\geq 0$ such that
$\DDD_{t+1}= \DDD_{t}$. Let us call this smallest $t$, the {\em nil-index} of
$\DDD$ and denote it by $\nil_{\xi} (I)$, or simply by $\nil$. As $\DDD$ is stable by the derivations
$\pp{\partial_{i, \xi}}$,
we easily check that for $t\geq \nil_{\xi} (I)$, $\DDD_{t}=\DDD$ and
that $\nil_{\xi} (I)$ is the maximal degree of the elements in $\DDD$.

\section{Deflation using first differentials}\label{Sec:Deflation}

\noindent To improve the numerical approximation of a root, one usually
applies a Newton-type methods to converge quadratically from
a nearby solution to the root of the system, provided it is simple.
In the case of multiple roots, deflation techniques are employed to
transform the system into another one which has an equivalent root
with a smaller multiplicity or even with multiplicity one.

We describe here a construction, using differentials of order one,
which leads to a system with a simple root. This construction improves
the constructions in \cite{lvz06,DaytonLiZeng11} 
since no new variables are added. 
It also improves the constructions presented in 
\cite{HauWam13} and the ``kerneling'' method of \cite{GiuYak13}
by adding a smaller number of equations at each deflation step.
Note that, in \cite{GiuYak13}, there are smart preprocessing and postprocessing
steps which could be utilized in combination with our method.  In the preprocessor, one
adds directly partial derivatives of polynomials which are zero at the root.  
The postprocessor extracts a square subsystem of the completely deflated system 
for which the Jacobian has full rank at the root.

\subsection{Determinantal deflation}
Consider the Jacobian matrix $J_{\fb} (\bx) = \left[\partial_{j} f_{i} (\bx)\right]$ of the
initial system $\fb$.
By reordering properly the rows and columns (i.e., polynomials and variables),
it can be put in the form 
\begin{equation} 
J_{\fb} (\bx) 
:= \left [
\begin{array}{cc}
A (\bx) & B (\bx) \\
C (\bx) & D (\bx)
\end{array}
\right] 
\end{equation}
where $A(\bx)$ is an $r\times r$ matrix with
$r = \rank J_{\fb}(\bxi) = \rank A(\bxi)$.

Suppose that $B(\bx)$ is an $r\times c$ matrix.  The $c$ columns
\begin{eqnarray*}\label{parkernel}
\det (A (\bx)) \left[
\begin{array}{c}
-A^{-1} (\bx) B (\bx) \\
\mathrm{Id}
\end{array}
\right]
\end{eqnarray*}
(for $r= 0$ this is the identity matrix) yield the $c$ elements 
$$ 
\Lambda_{1}^{\bx}=\sum_{i=1}^{n} \lambda_{1,j} (\bx) \partial_{j},~\ldots,~
\Lambda_{c}^{\bx}=\sum_{i=1}^{n}\lambda_{c,j} (\bx) \partial_{j}.
$$
Their coefficients $\lambda_{i,j}(\bx)\in \KK[\bx]$ are polynomial in the
variables $\bx$. 
Evaluated at $\bx=\bxi$, they generate the kernel of
$J_{\fb} (\bxi)$ and form a basis of $\DDD_{1}$. 

\begin{definition}
The family $D^{\bx}_{1}=\{\Lambda_{1}^{\bx}, \ldots,
\Lambda_{c}^{\bx}\}$ is the {\em formal} inverse system of
order $1$ at $\bxi$. 
For $\bm i=\{i_{1},\ldots, i_{k}\}\subset$ $\{1, \ldots, c\}$ 
with $|\bm i|\neq 0$, the $\bm i$-{\em deflated system} of order~$1$~of~$\fb$~is
$$ 
\{\fb, {\Lambda}_{i_{1}}^{\bx} (\fb), \ldots, {\Lambda}_{i_{k}}^{\bx} (\fb)\}.
$$
\end{definition}
The deflated system is obtained by adding some minors of the Jacobian
matrix $J_{\fb}$ as shown by the following lemma. Note that this establishes the close relationship of our method 
to the {\em isosingular deflation} involved in \cite{HauWam13}.
\begin{proposition}\label{determ}
 For $i=1, \ldots, c$, 
\begin{equation}\label{eq:minor}
\Lambda^{\bx}_{i}(f_{j})= 
\left |
  \begin{array}{cccc}
\partial_{1} f_{1}& \cdots&   \partial_{r} f_{1}  & 
            \partial_{r+i} f_{1}\\ 
            \vdots & &\vdots & \vdots\\ 
\partial_{1} f_{r}& \cdots&   \partial_{r} f_{r}  &  
            \partial_{r+i} f_{r}\\ 
  \partial_{1} f_{j}& \cdots&   \partial_{r} f_{j} & \partial_{r+i} f_{j}
\end{array}
\right|.
\end{equation}
\end{proposition}
\begin{proof}
We have
$
\Lambda_{i}^{\bx}(f_{j})
= \sum_{k=1}^{r} \lambda_{i,k}  \partial_{k}(f_{j}) + \det(A)\, \partial_{r+i} (f_{j}) 
$
where ${\blambda}=[\lambda_{i,1},\ldots, \lambda_{i,r}]$ $= - \det(A)\, 
A^{-1} B_{i}$ is the solution of the system 
$$ 
A \, {\blambda} + \det(A) B_{i} = 0,
$$
and $B_{i}$ is the $i^{th}$ column of $B$. By Cramer's rule,
$\lambda_{i,k}$ is up to $(-1)^{r+k+1}$ the $r\times r$ minor of the matrix $[A \, | \,
B_{i}]$ where the $k^{th}$ column is removed.
Consequently $\Lambda_{i}^{\bx}(f_{j}) =\sum_{k=1}^{r} \lambda_{i,k} (\bx) \partial_{k}(f_{j}) +
\det(A)\, \partial_{r+i} (f_{j})$ corresponds to the expansion of the
determinant \eqref{eq:minor} along the last row.
\end{proof}
This proposition implies that ${\Lambda}_{i}^{\bx} (\fb)$
has at most $n-c$ zero entries ($j\not \in [1,\ldots, r]$). Thus,
the number of non-trivial new equations added in the
$\bm i$-deflated system
is \mbox{$|\bm i|\cdot(N-n+c)$}.
The construction depends on the choice of the invertible block $A(\bxi)$ in
$J_{\fb} (\bxi)$. 
By a linear invertible transformation of the initial system and by
computing a $\bm i$-deflated system, one obtains 
a deflated system constructed from any $|\bm i|$ linearly 
independent elements of the~kernel~of~$J_{\fb} (\bxi)$.

\begin{example}\label{Ex:Illustrative} Consider the multiplicity $2$ root $\bxi = (0,0)$ 
for the system $f_1(\bx) = x_1 + x_2^2$ and $f_2(\bx) = x_1^2 + x_2^2$.
Then,
{\small $$
J_{\fb}(\bx) =
\left[
\begin{array}{cc}
A (\bx) & B (\bx) \\
C (\bx) & D (\bx)
\end{array}
\right] = \left[\begin{array}{cc} 1 & 2x_2 \\ 2x_1 & 2x_2 \end{array}\right].
$$}
% \negskip%
As $A(\bxi)$ is of rank $1$, the $\{1\}$-deflated system of 
order $1$ of $\fb$ obtained by adding the $2\times 2$ bording minor of
$A$, that is the
determinant of the $J_{\fb}$, is
$$
\left\{x_1 + x_2^2, ~~x_1^2 + x_2^2, ~-4x_1x_2 + 2x_2\right\},
$$
which has a multiplicity $1$ root at $\bxi$.
\end{example}
 
We use the following to analyze this deflation procedure.

\begin{lemma}[Leibniz rule]
For $a,b\in \KK[\bm x]$, 
$$ 
\bpartial^{\alpha} (a\,b) =\sum_{\beta\in \N^{n}} \frac{1}{\beta!} \bpartial^{\beta} (a) \pp{\partial}^{\bm\beta}
(\bpartial^{\alpha}) (b).
$$
\end{lemma}

\begin{proposition}\label{deflation:1} Let $r$ be the rank of
  $J_{\fb}(\bxi)$. Assume that $r<n$. Let $\bm i\subset
  \{1,\ldots,n\}$ with $0<|\bm i|\leq n-r$ 
  and $\fb^{(1)}$~be the $\bm i$-{\em deflated system} of order $1$ of
$\fb$. Then, $\mult_{\bxi}
(\fb^{(1)})\geq 1$ and $\nil_{\bxi} (\fb^{(1)}) < \nil_{\bxi} (\fb)$, which also implies that $\mult_{\bxi}
(\fb^{(1)})< \mult_{\bxi}
(\fb)$.
\end{proposition}
\begin{proof}
By construction, for $i\in\bm i$, 
the polynomials ${\Lambda}_{i}^{\bx} (\fb)$
vanish at $\bxi$, so that $\mult_{\bxi} (\fb^{(1)})\ge 1$. 
By hypothesis, the Jacobian of $\fb$ is not 
injective yielding $\nil_{\bxi}(\fb)> 0$.
Let $\DDD^{(1)}$ be the inverse
system of $\fb^{(1)}$ at $\bxi$. 
Since $(\fb^{(1)})\supset (\fb)$, 
we have $\DDD^{(1)}\subset \DDD$.
In particular, for any non-zero element $\Lambda \in \DDD^{(1)}\subset
\KK[\bm\partial_{\bxi}]$ and $i\in \bm i$, 
$\Lambda (\fb)=0$ and $\Lambda ({\Lambda}^{\bx}_{i}
(\fb))=0$.  

Using Leibniz rule, for any $p\in \KK[\bx]$, we have 
{%\scriptsize
\begin{eqnarray*}
\Lambda ({\Lambda}_{i}^{\bx} (p)) &=&
\Lambda \left(\sum_{j=1}^{n} {\lambda}_{i,j} (\bx)\partial_{j} (p)\right)\\
&=&
\sum_{\beta\in \N^{n}} \sum_{j=1}^{n} \frac{1}{\beta!} \bpartial_{\bxi}^{\bm\beta} ( {\lambda}_{i,j} (\bx)) 
\pp{\partial_{\bxi}}^{\bm\beta}
(\Lambda) \partial_{j,\bxi} (p)\\ 
&=&
\sum_{\beta\in \N^{n}} \sum_{j=1}^{n} \frac{1}{\beta!} \bpartial_{\bxi}^{\bm\beta} ( {\lambda}_{i,j} (\bx)) \bpartial_{j,\xi} \pp{\partial_{\xi}}^{\bm\beta}
(\Lambda) (p)\\
&=&
\sum_{\beta\in \N^{n}} \Delta_{i,\beta} \pp{\partial_{\xi}}^{\bm\beta} (\Lambda) (p)\\
\end{eqnarray*}
} where 
{%\small
$$ 
\Delta_{i,\bm\beta}=\sum_{j=1}^{n}
{\lambda}_{i,j,\bm\beta} \partial_{j,\xi}\in\KK[\bm\partial_{\bxi}]
\hbox{~and~}
{\lambda}_{i,j,\bm\beta} = \frac{1}{\bm\beta!}
\partial_{\bxi}^{\bm\beta}({\lambda}_{i,j}(\bx))\in \KK.$$
}

The term $\Delta_{i,\bm 0}$ is
$\sum_{j=1}^{n}
\lambda_{i,j} (\bxi) \partial_{j,\xi}$ which has 
degree $1$ in~$\bpartial_{\bxi}$ 
since $[\lambda_{i,j} (\bxi)]$ is a non-zero element of
$\ker J_{\fb} (\bxi)$.
For simplicity, let $\phi_{i}(\Lambda):= \sum_{\bm\beta\in \N^{n}} \Delta_{i,\bm\beta}
\pp{\partial}^{\bm\beta} (\Lambda)$.

For any $\Lambda\in \C[\bpartial_{\bxi}]$, we have
{%\small
\begin{eqnarray*}
\pp{\partial_{j,\xi}} (\phi_{i}(\Lambda)) 
&=&\sum_{\beta\in \N^{n}} \lambda_{i,j,\beta}   \pp{\partial}^{\bm\beta}(\Lambda)+
\Delta_{i,\beta}  \pp{\partial}^{\bm\beta} (\pp{\partial_{j,\xi}}
(\Lambda))\\
&=&\sum_{\beta\in \N^{n}} \lambda_{i,j,\beta}   \pp{\partial}^{\bm\beta}(\Lambda)+
\phi_{i} (\pp{\partial_{j,\xi}} (\Lambda)).
\end{eqnarray*} 
} Moreover, if $\Lambda \in \DDD^{(1)}$, then by definition
$\phi_{i}(\Lambda) (\fb)=0$. 
Since $\DDD$ and $\DDD^{(1)}$ are both stable by derivation,
it follows that $\forall \Lambda \in \DDD^{(1)}$,
$\pp{\partial_{j,\xi}} (\phi_{i} (\Lambda))\in \DDD^{(1)}+ \phi_{i}(\DDD^{(1)})$. 
Since \mbox{$\DDD^{(1)}\subset \DDD$}, we know $\DDD+\phi_{i} (\DDD^{(1)})$ is stable by
derivation.  For any element $\Lambda$ of $\DDD+\phi_{i} (\DDD^{(1)})$,
$\Lambda (\fb)=0$.  We deduce that $\DDD+\phi_{i} (\DDD^{(1)})=\DDD$.
Consequently, the order of the elements in $\phi_{i} (\DDD^{(1)})$ 
is at most $\nil_{\bxi} (\fb)$.
The statement follows since $\phi_i$ increases the order by $1$,
therefore $\nil_{\bxi}(\fb^{(1)})< \nil_{\bxi}(\fb)$.
\end{proof}

We consider now a sequence of deflations of the system~$\fb$.
Let $\fb^{(1)}$ be the ${\bm i}_{1}$-deflated system of $\fb$. We
construct inductively  
$\fb^{(k+1)}$ as the ${\bm i}_{k+1}$-deflated system of $\fb^{(k)}$ for some
choices of ${\bm i}_{j}\subset \{1,\ldots,n\}$. 
 
\begin{proposition}
There exists $k\leq \nil_{\bxi} (\fb)$ such that $\bxi$ is a
simple root of $\fb^{(k)}$.
\end{proposition}
\begin{proof}
By Proposition \ref{deflation:1}, $\mult_{\bxi} (\fb^{(k)}) \geq 1$ and $\nil_{\bxi} (\fb^{(k)})$ is
strictly decreasing with $k$ until it reaches the value $0$.
Therefore, there exists $k\leq \nil_{\bxi}(I)$ such that $\nil_{\bm\xi}
(\fb^{(k)})=0$ and $\mult_{\bxi} (\fb^{(k)})\geq~1$.
This implies that $\bxi$ is a simple root of $\fb^{(k)}$.
\end{proof}
 
To minimize the number of equations added at each deflation step, we
take $|\bm i|=1$. Then, the number of non-trivial 
new equations added at each step is at most $N-n+c$.

Here, we described an approach using first order differentials 
arising from the Jacobian, but this 
can be easily extended to use higher order differentials.

\subsection{Isosingular decomposition}\label{isosing}

As presented above, the $\bm i$-deflated system
can be constructed even when $\bxi$ is not isolated.
For example, let $\fb^{(1)}$ be the resulting system
if one takes $\bm i =  \{1,\dots,c\}$.
Then, $\fb^{(1)}(\bx) = 0$ if and only if 
$f(\bx) = 0$ and either $\rank\, J_\fb(\bx) \leq r$
or $\det A(\bx) = 0$.  If $\det A(\bx) \neq 0$,
then this produces a {\em strong deflation} in the sense of \cite{HauWam13}
and thus the results of~\cite{HauWam13} involving 
isosingular deflation apply directly to this new~deflation~approach.

One result of \cite{HauWam13} is a stratification
of the solution set of $\fb = 0$, 
called the {\em isosingular decomposition}.
This decomposition produces a finite collection
of irreducible sets $V_1,\dots,V_k$
consisting of solutions of $\fb = 0$,
called {\em isosingular sets} of~$\fb$, 
i.e. Zariski closures of sets of points with the same determinantal deflation sequence
(see \cite[Definition 5.1]{HauWam13} for the precise definition of isosingular sets).
Rather than use the isosingular deflation 
of \cite{HauWam13} which deflates using 
all minors of $J_\fb(x)$ of size $(r+1)\times(r+1)$
where $r = \rank\, J_\fb(\bxi)$, 
one can utilize the approach above 
with $\bm i = \{1,\dots,c\}$.
If $\det A(\bx)\neq 0$ on the solution set, then one obtains directly
the components of the isosingular decomposition.
Otherwise, one simply needs to further investigate
the components which arise with $\det A(\bx) = 0$.  

We describe this computation in detail using two examples.
In the first example, \mbox{$\det A(\bx) = 1$} so that the method
applies directly to computing an isosingular decomposition.  
In the second, we show how to handle the case-by-case analysis
when $\det A(\bx)$ could be zero.

\begin{example}
Consider the polynomial system $\fb(x,y,z)$ where
$$f_1 = x-y^2,~~ f_2 = x+y^2z,~~ f_3 = x^2 - y^3 - xyz.$$
By \cite[Thm.~5.10]{HauWam13},
every isosingular set of $\fb$ is either an irreducible
component of the solution set $\fb = 0$ or is an irreducible
component of the singular set of an isosingular set.  
We start by computing the irreducible components of $\fb = 0$,
namely $V_1 = \{x = y = 0\}$.

Since the curve $V_1$ has multiplicity $2$ with respect to $\fb$,
we need to deflate.  Since the Jacobian
$$J_\fb = \left[\begin{array}{ccc}
 1 &          -2y &    0 \\
1 &         2yz &  y^2 \\
2x - yz& -3y^2 - xz & -xy
\end{array}\right]$$
has rank $1$ on $V_1$, isosingular deflation
would add in all $9$ of the $2\times 2$ minors of $J_\fb$.
This would guarantee that all solutions of the resulting
deflated system would have $\rank J_\fb = 1$
since $J_\fb$ can never be the zero matrix.
However, by using the approach above, we only add 
$4$ polynomials:
$$\fb^{(1)} = \{\fb,  
~~2y + 2yz,~~ 2y(2x - yz) - xz - 3y^2,~~y^2,~-xy\}.$$
Moreover, since $A = 1$, which is the upper left corner of $J_\fb$, 
we obtain the same condition 
as above with the deflation $\fb^{(1)}$, i.e., $\fb^{(1)} = 0$ if and 
only if $\fb = 0$ and $\rank J_\fb = 1$.  
Moreover, one can easily verify that 
$V_1$ has multiplicity $1$ with respect to $\fb^{(1)}$,
i.e., $J_{\fb^{(1)}}$ generically has rank $2$ on $V_1$.

The next step is to compute all points on $V_1$
where $J_{\fb^{(1)}}$ has rank at most $1$.  Since
$$J_{\fb^{(1)}}(0,0,z) = \left[\begin{array}{ccc}
1 & 0 & 0 \\
1 & 0 & 0 \\
0 & 0 & 0 \\
0 & 2z + 2 & 0 \\
-z & 0 & 0 \\
0 & 0 & 0 \\
0 & 0 & 0
\end{array}\right],$$
one observes that the point $(0,0,-1)$ is isosingular with
respect to $\fb$.  Therefore, the irreducible sets 
$V_1$ and $V_2 = \{(0,0,-1)\}$ form the isosingular decomposition of $\fb$.

Since $\bxi = (0,0,-1)$ is an isosingular point, deflation
will produce a system for which this point is nonsingular.  
To that end, since $\rank J_{\fb^{(1)}}(\bxi) =1$, i.e., $c = 2$,
we can use the same null space used in the construction of $\fb^{(1)}$.
In particular, the next deflation adds at most $8$ polynomials.  In
this case, two of them are identically zero so that $\fb^{(2)}$
consists of $13$ nonzero polynomials, $11$ of which are distinct,
with~$\bxi$ being a nonsingular root.  If one instead
used isosingular deflation with all minors, the resulting 
deflated system would consist of $139$ distinct polynomials.  
\end{example}

\begin{example}
Consider the polynomial system $\fb(w,x,y,z)$ where
$$f_1 = w^2 - y^2 - x^3 - yz, ~~ f_2 = z^2.$$
The solution set of $\fb = 0$ is the irreducible cubic
surface 
$$V_1 = \{(w,x,y,0)~|~y^2 = w^2 - x^3\}.$$
Since $V_1$ has multiplicity $2$ with respect to $\fb$,
we deflate by using $A = 2w$ to yield $\fb^{(1)} = \{\fb,~4wz\}$.

Next, we consider the set of points on $V_1$ where
$\rank J\fb^{(1)}\leq 1$.  Since
$$J_{\fb^{(1)}}(w,x,y,0) = \left[\begin{array}{cccc}
2w & -3x^2 & -2y & -y \\
0 & 0 & 0 & 0 \\
0 & 0 & 0 & 4w
\end{array}\right],$$
$\rank J_{\fb^{(1)}}\leq 1$ on the curve
$C = V_1\cap\{w = 0\} = \{(0,x,y,0)~|~y^2 = x^3\}$.
However, since $A = 2w$ is identically zero on this 
curve, we are not guaranteed that this curve is an isosingular
set of $\fb$.  One simply checks if it is an isosingular
set by deflating the original system $\fb$ on this curve.  
If one obtains the curve $C$, then it is an isosingular set and one 
proceeds as above.  Otherwise, the generic points of $C$ 
are smooth points with respect to $\fb$ on a larger isosingular set,
in which case one uses the new deflation to compute new candidates
for isosingular sets.  

To deflate $C$ using $\fb$, we take $A = -y$, the top right corner of $J_\fb$, to yield
$$\gb^{(1)} = \{\fb,
~-4wz,~6x^2z,~2z(2y + z)\}.$$
Since $C\subset V_1$ and $J\gb^{(1)}$ generically has
rank $2$ on $C$ and $V_1$, we know that $C$ is
not an isosingular set with respect to $\fb$.  However,
this does yield information about the isosingular components of $\fb$,
namely there are no curves and each isosingular point must be contained in $C$.
Hence, restricting to $C$, one sees that $\rank J_{\gb^{(1)}}(\bxi)\leq 1$ 
if and only if $\bxi = (0,0,0,0)$.  
Since $\gb^{(1)}$ was constructed using $A = -y$ 
which vanishes at this point, we again need to verify
that the origin is indeed an isosingular point, i.e., deflation
produces a system for which the origin is a nonsingular root.
To that end, since $J_{\fb}(\bxi) = 0$, the first deflation
simply adds  all partial derivatives.  The Jacobian
of the resulting system has rank $3$ for which
one more deflation regularizes~$\bxi$.
Therefore, $V_1$ and $V_2 = \{(0,0,0,0)\}$
form the isosingular decomposition of~$\fb$.
\end{example}

\section{The multiplicity structure}\label{Sec:PointMult}
Before describing our results, we start this section by recalling the  definition of orthogonal primal-dual  pairs of bases for the space $\K[\bx ]/Q$ and its dual. The following is a definition/lemma: 

\begin{lemma}[Orthogonal primal-dual basis pair]\label{pdlemma}
Let $\fb$, $\bxi$,  $Q$, $\DDD$, $\mult= \mult_\xi(\fb)$ and $\nil=\nil_\xi(\fb)$ be as in the Preliminaries. 
Then there exists a primal-dual basis pair of the local ring  $\K[\bx ]/ \QQ$ with the following properties:
\begin{enumerate}
\item The {\em primal basis} of the local ring  $\K[\bx ]/ \QQ$ has the form 
\begin{equation}\label{pbasis}
B:=\left\{( \bx-\xi)^{\balpha_0},  ( \bx-\xi)^{\balpha_1},\ldots, ( \bx-\xi)^{\balpha_{\mult-1}}\right\}.
\end{equation}  
We can assume that  $\alpha_0=0$ and that the monomials in $B$ are  {\em connected to 1}
(c.f. \cite{Mourrain99-nf}). Define the set of exponents in $B$ 
\begin{eqnarray}\label{E}
E:=\{\alpha_0, \ldots, \alpha_{\mult-1}\}.
\end{eqnarray} 
\item The unique {\em dual basis} $\bLambda=\{ \Lambda_{0},
  \Lambda_{1},\ldots$, $\Lambda_{{\mult-1}} \}\subset \DDD$
  orthogonal to $B$ has the form:
\begin{eqnarray}\label{Macbasis}
\Lambda_{0}&=& \bpartial^{\balpha_0}_\bxi=1_\bxi\nonumber \\
\Lambda_{1}&=&\frac{1}{\balpha_1 !}\bpartial_\bxi^{\balpha_1}  +\sum_{|\bbeta|\leq |\balpha_1| \atop \bbeta\not\in E}\nu_{\balpha_1, \bbeta} \;\frac{1}{\bbeta !}\bpartial_\bxi^{\bbeta}\nonumber\\
&\vdots&\\
\Lambda_{\mult-1}&=&\frac{1}{\balpha_{\mult-1} !}\bpartial_\bxi^{\balpha_{\mult-1}}  +\sum_{|\bbeta|\leq  |\balpha_{\mult-1}|\atop \bbeta\not\in E}\nu_{\balpha_{\mult-1}, \bbeta}\; \frac{1}{\bbeta !}\bpartial_\bxi^{\bbeta},\nonumber
\end{eqnarray} 
\item We have 
$0=\ord(\Lambda_{0}) \leq \cdots \leq \ord(\Lambda_{\mult-1})$, and  for  all    $0\leq t\leq \nil$ we have 
$$
\DDD_t={\rm span}\left\{ \Lambda_{j}\;:\;  \ord(\Lambda_{{j}})\leq t \right\},
$$ where $\DDD_{t}$ denotes the elements of $\DDD$ of order $\leq t$, as above.
\end{enumerate}
\end{lemma}

\begin{proof}
  Let $\succ$ be any graded monomial  ordering in $\K[\bpartial]$.
  We consider the initial $\mathrm{In}(\DDD)=\{\mathrm{In}(\Lambda)\mid \Lambda \in \DDD\}$ of $\DDD$ for the monomial ordering $\succ$.
  It is a finite set of increasing monomials
$D:=\left\{\bpartial^{\balpha_0},  \bpartial^{\balpha_1},\ldots, \bpartial^{\balpha_{\mult-1}}\right\},$
which are the leading monomials of the elements of a basis 
$\bLambda=\{\Lambda_{0},  \Lambda_{1},\ldots$, $\Lambda_{{\mult-1}}\}
$ of $\DDD$.
As $1\in \DDD$ and is the lowest monomial $\succ$, we have $\Lambda_{0}=1$.
As $\succ$ is refining the total degree in $\K[\bpartial]$, we have $\ord({\Lambda}_{i})=|\balpha_{i}|$ and
$0=\ord({\Lambda}_{0}) \leq \cdots \leq \ord({\Lambda}_{\mult-1})$.
Moreover, every element in $\DDD_{t}$ reduces to $0$ by the elements in $\bLambda$.
As only the elements $\Lambda_{{i}}$ of order $\le t$ are involved in this reduction, we deduce that
$\DDD_{t}$ is spanned by the elements $\Lambda_{{i}}$ with $\ord({\Lambda}_{i})\leq t$.

Let $E=\{ \balpha_{0},\ldots, \balpha_{\mult-1}\}$.
The elements $\Lambda_{{i}}$ can be written in the form
$$
\Lambda_{{i}}=\frac{1}{\balpha_{i} !}\bpartial_\bxi^{\balpha_{i}}  +\sum_{|\bbeta|\prec |\balpha_{i}|}\nu_{\balpha_i, \bbeta}\; \frac{1}{\bbeta !}\bpartial_\bxi^{\bbeta}.
$$
By auto-reduction of the elements  $\Lambda_{{i}}$, we can even suppose that $\bbeta\not\in E$ in the summation above, so that they are of the form \eqref{Macbasis}.

Let ${B}=\left\{( \bx-\xi)^{{\balpha}_0}, \ldots, ( \bx-\xi)^{{\balpha}_{\mult-1}}\right\}\subset \K[\bx]
$. As $(\Lambda_{i}((\bx-\xi)^{{\balpha}_j}))_{0\le i,j\leq \mult-1}$ is the identity matrix, we deduce that
$B$ is a basis of $\K[\bx ]/ \QQ$, which is dual to $\bLambda$.

As $\DDD$ is stable by derivation, the leading term of $\frac{d }{d\partial_{i}}(\Lambda_{j})$ is in $D$.
If $\frac{d }{d\partial_{i}}(\bpartial_\bxi^{\balpha_{j}})$ is not zero, then it is the leading term of
$\frac{d }{d\partial_{i}}(\Lambda_{j})$, since the monomial ordering is compatible with the
multiplication by a variable. This shows that $D$ is stable by division by the variable $\partial_{i}$
and that $B$ is connected to $1$. This completes the proof.  
\end{proof}
 
A basis $\bLambda$ of $\DDD$  as described in Lemma \ref{pdlemma} can be obtained from any other basis $\tilde{\bLambda}$ of ${\DDD}$ by first choosing pivot elements that are the leading monomials with respect to a degree monomial ordering on $\K[\bpartial]$, then transforming  the coefficient matrix of $\tilde{\bLambda}$ into row echelon form using the pivot leading coefficients.
The integration method described in \cite{mantzaflaris:inria-00556021} computes a primal-dual pair
such that the coefficient matrix has a block row-echelon form, each block being associated to an order.
The computation of a basis as in Lemma \ref{pdlemma} can be then performed order by order.

\begin{example}\label{ex:mult3} Let 
$$f_1=x_1-x_2+x_1^2, f_2= x_1-x_2+x_1^2,
$$
which has a multiplicity $3$ root at $\bxi=(0,0)$.  The integration method described in~\cite{mantzaflaris:inria-00556021} computes a primal-dual pair 
$$
\tilde{B}=\left\{1,x_1,x_2\right\}, \; \tilde{\bLambda}=\left\{1, \partial_1+\partial_2, \partial_2+\frac{1}{2}\partial_1^2+ \partial_1\partial_2+\frac{1}{2}\partial_1^2\right\}.
$$
This primal dual pair does not form an orthogonal pair, since $( \partial_1+\partial_2)(x_2)\neq 0$. However, using let say the degree lexicographic ordering such that $x_1>x_2$, we easily deduce the primal-dual pair of Lemma \ref{pdlemma}:
$$
{B}=\left\{1,x_1,x_1^2\right\}, \quad \bLambda=\tilde{\bLambda}.
$$
\end{example}
\vspace{2mm}

Throughout this section we assume that we are given a fixed primal basis $B$ for $\K[\bx ]/ \QQ$
such that a dual basis $\bLambda$ of $\DDD$ satisfying the properties of Lemma \ref{pdlemma} exists. Note that such a primal basis $B$  can be computed numerically  from an approximation of $\xi$ and using a modification of the integration method of  \cite{mantzaflaris:inria-00556021}.

A  dual basis can also be computed by
Macaulay's dialytic method which can be used to deflate the root $\bxi$ as
in \cite{lvz08}. This method would introduce
\mbox{$n+(\mult-1)
\left({{n+\nil}\choose{n}}-\mult\right)$} new variables, which is not polynomial in $\nil$. Below, we give a construction of a 
polynomial system that only depends on at most 
$n+ n\mult(\mult-1)/2$ variables. These variables 
correspond to the entries of the {\em multiplication matrices} that we~define~next.
Let 
\begin{eqnarray*}
M_{i} : \K[\bx]/Q&\rightarrow &  \K[\bx]/Q\\
  p & \mapsto & (x_{i}-\xi_{i})\, p
\end{eqnarray*} 
be the multiplication operator by $x_{i}-\xi_{i}$ in 
$\K[\bx]/Q$. Its transpose operator is
\begin{eqnarray*}
M_{i}^{t} : \DDD&\rightarrow &  \DDD\\
  \Lambda & \mapsto & \Lambda \circ M_{i}= (x_{i}-\xi_{i})\cdot \Lambda = \frac{d}{d\partial_{i,\xi}} (\Lambda)=d_{\partial_{i, \xi}}(\Lambda),
\end{eqnarray*} 
where $\DDD= Q^{\perp}\subset \K[\bpartial]$. The matrix of
$M_{i}$ in the basis $B$ of  $\K[\bx]/Q$ is denoted $\mM_{i}$.

As $B$ is a basis of $\K[\bx]/Q$, we can identify the elements of
$\K[\bx]/Q$ with the elements of the vector space $\sp_\K( B)$. 
We define the normal form $N(p)$ of a polynomial $p$ in $\K[\bx]$ as the
unique element $b$ of ${\rm span}_\K(B)$ such
that $p-b\in Q$. Hereafter, we are going to identify the elements of
$\K[\bx]/Q$ with their normal form in $\sp_\K (B )$.

For any polynomial $q (x_{1}, \ldots, x_{n}) \in \K[\bx]$, we denote by  $q
(\xi+\bM)$ be the operator on $\K[\bx]/Q$ obtained by replacing $x_{i}-\xi_{i}$ by $M_{i}$, i.e. it is defined as
$$q(\xi+\bM):= \sum_{\bgamma\in \N^n} \frac{1}{\bgamma!}  \bpartial_{\xi}^{\bgamma}(q) \, \bM^{\bgamma}, $$
using the notation  $\bM^{\bgamma}:=M_1^{\gamma_1}\circ\cdots\circ M_n^{\gamma_n}$.  Similarly,  we denote by 
$$q(\xi+\mM):= \sum_{\bgamma\in \N^n} \frac{1}{\bgamma!}\bpartial_{\xi}^{\bgamma}(q)  \, \mM^{\bgamma}, $$
the matrix of $q(\xi+\bM)$ in the basis $B$ of $\K[\bx]/Q$, where  $\mM^{\bgamma}:=\mM_1^{\gamma_1}\cdots \mM_n^{\gamma_n}$.
Note that the operators $ \{M_i\}$ and the multiplication matrices $\{\mM_i\}$ are pairwise commuting. 

\begin{lemma}\label{lem:normalform}
For any $q\in \K[\bx]$, the normal form of $q$ is $N (q)= q (\xi+\bM) (1)$ and we
have
$$ 
q (\xi+\bM) (1) = \Lambda_{0}(q)\, 1 + \Lambda_{1}(q) \, ( \bx-\bxi)^{\balpha_1}+\cdots + \Lambda_{{\mult-1}}(q)\, ( \bx-\bxi)^{\balpha_{\mult-1}}.
$$
\end{lemma}

\begin{proof}
We have $q(\xi+\bM)(1)=q\!\mod Q=N(q)$. The second claim follows from the orthogonality of $\bLambda$ and $B$. 
\end{proof}

This shows that the coefficient vector $[p]$ of $N (p)$ in the basis $B$ of
 is $[p]= (\Lambda_{{i}}(p))_{0\le i \le \mult-1}$.
 
The  following lemma is also well known, but we include it here with proof:
\begin{lemma}\label{multlemma} Let $B$ as in (\ref{pbasis}) and denote the exponents in $B$ by 
$
E:=\{\alpha_0, \ldots, \alpha_{\mult-1}\}
$ as above.
Let  
$$E^+:=\bigcup_{i=1}^n (E+ \e_i )$$ 
with $E+\e_i=\{(\gamma_1, \ldots , \gamma_i+1, \ldots, \gamma_n):\gamma\in E\}$ and 
we denote $\partial(E)= E^{+}\setminus E$. 
The values of the coefficients $\nu_{\alpha, \beta}$
for $(\alpha,\beta)\in E\times \partial(E)$ appearing in the dual basis (\ref{Macbasis})
uniquely determine the system of pairwise commuting multiplication  matrices $\mM_{i}$, namely,  
for $i=1, \ldots, n$ 
\begin{eqnarray}\label{multmat}
\mM_{i}^{t}=
\begin{array}{|ccccc|}
\cline{1-5}
 0&\nu_{\balpha_1, \e_i}&\nu_{\balpha_2, \e_i}&\cdots &\nu_{\balpha_{\mult-1}, \e_i}  \\ 
 0&0&\nu_{\balpha_2,\balpha_1+\e_i}&\cdots &\nu_{\balpha_{\mult-1},\balpha_1+\e_i}\\  
 \vdots &\vdots &&&\vdots\\
0&0&0&\cdots &\nu_{\balpha_{\mult-1},\balpha_{\mult-2}+\e_i}\\ 
0&0&0&\cdots &0\\
\cline{1-5} 
\end{array}
\end{eqnarray} 
Moreover, 
$$
\nu_{\alpha_i, \alpha_k+\e_j}=\begin{cases} 1 & \text{ if } \alpha_i= \alpha_k+\e_j\\
0 & \text{ if } \alpha_k+\e_j \in E, \; \alpha_i \neq \alpha_k+\e_j .
\end{cases}
$$
\end{lemma}
\begin{proof}
As $M_{i}^{t}$ acts as a derivation on $\DDD$ and $\DDD$ is closed under derivation,  so the third property in Lemma \ref{pdlemma} implies that the matrix
of $M_{i}^{t}$ in the basis $\bLambda=\{\Lambda_0, \ldots, \Lambda_{\mult-1}\}$ of $\DDD$ has an upper triangular form with
zero (blocks) on the diagonal.

For an element $\Lambda_{{j}}\in \bLambda $ of order $k$, its image by
$M_{i}^{t}$ is 
{%\small
  \begin{eqnarray*}
&&M_{i}^{t} (\Lambda_{{j}}) =
(x_{i} - \xi_{i}) \cdot \Lambda_{{j}}\\
&&=\sum_{|\balpha_{l}|<k} \Lambda_{{j}} ((x_{i}-\xi_{i})
(\bx-\bxi)^{\balpha_{l}}) \Lambda_{{l}}\\
 &&=  \sum_{|\balpha_{l}|<k} \Lambda_{{j}}
((\bx-\bxi)^{\balpha_{l}+\e_{i}}) \, \Lambda_{{l}}
 =  \sum_{|\balpha_{l}|<k} 
\nu_{\balpha_{j}, \balpha_l+ \e_i} \Lambda_{{l}}.
\end{eqnarray*} }
This shows that the entries of $\mM_{i}$ are the coefficients of the
dual basis elements corresponding to exponents in $E\times \partial(E)$. The second claim is clear from the definition of $\mM_{i}$.
\end{proof}

The previous lemma shows that the dual basis uniquely defines the system of multiplication matrices for $  i=1, \ldots, n$
{%\small 
\begin{eqnarray*}
\mM_{i}^{t}&=&\begin{array}{|ccc|}
\cline{1-3}
 \Lambda_{{0}}(x_i-\xi_i)&\cdots &\Lambda_{{\mult-1}}(x_i-\xi_i)  \\ 
 \Lambda_{{0}}\left(( \bx-\bxi)^{\balpha_1+\e_i}\right)&\cdots &\Lambda_{{\mult-1}}\left(( \bx-\bxi)^{\balpha_1+\e_i}\right) \\  
 \vdots &&\vdots\\
 \Lambda_{{0}}\left(( \bx-\bxi)^{\balpha_d+\e_i}\right)&\cdots &\Lambda_{{\mult-1}}\left(( \bx-\bxi)^{\balpha_\mult+\e_i}\right) \\  
\cline{1-3} 
\end{array}\nonumber\\
&=& 
\begin{array}{|ccccc|}
\cline{1-5}
  0&\nu_{\balpha_1, \e_i}&\nu_{\balpha_2, \e_i}&\cdots &\nu_{\balpha_{\mult-1}, \e_i}  \\ 
 0&0&\nu_{\balpha_2,\balpha_1+\e_i}&\cdots &\nu_{\balpha_{\mult-1},\balpha_1+\e_i}\\  
 \vdots &\vdots &&&\vdots\\
0&0&0&\cdots &\nu_{\balpha_{\mult-1},\balpha_{\mult-2}+\e_i}\\ 
0&0&0&\cdots &0\\
\cline{1-5} 
\end{array}
\end{eqnarray*}}
Note that these matrices are nilpotent by their upper triangular
structure, and all $0$ eigenvalues. 
As $\nil$ is the maximal order of the elements of $\DDD$, we have
$\mM^{\bgamma}=0$ if $|\bgamma|> \nil$.\\

% The proof also highlights that the pairwise commutativity of the
% multiplication matrices  follows from the fact that the Macaulay dual
% space is closed under integration. \\ 

Conversely, the system of multiplication matrices $\mM_1, \ldots, \mM_n$ uniquely defines the dual basis as follows. Consider $\nu_{\balpha_i,\bgamma}$  for some  $(\balpha_i, \bgamma)$ such that $|\bgamma|\leq \nil$ but $\bgamma \not\in E^+$.  We can uniquely determine $\nu_{\balpha_i,\bgamma}$ from the values of $\{\nu_{\balpha_j, \bbeta}\; : \;(\balpha_j, \bbeta)\in E\times \partial(E)\}$  from the following identities:
\begin{equation}\label{restnu}
\nu_{\balpha_i,\bgamma}= \Lambda_{i}((\bx -\bxi)^{\bgamma})
=[\mM_{(\bx -\xi)^{\bgamma}}]_{i, 1}= [\mM^{\bgamma}]_{i,1}.
\end{equation}

The next definition defines the {\em parametric multiplication matrices} that we use  in our constriction.

\begin{definition}[Parametric multiplication matrices] Let   $E$, $\partial(E)$ as in Lemma \ref{multlemma}.  We define an array  ${\bmu}$ of length $n\mult(\mult-1)/2$ consisting of $0$'s, $1$'s and the variables 
  $\mu_{\alpha_i, \beta}$
  %as in  (\ref{mudef})
  as follows: for all $\alpha_i, \alpha_k\in E$ and $j\in \{1, \ldots, n\}$ the corresponding entry is 
\begin{eqnarray}\label{defmuE}
{\bmu}_{\alpha_i, \alpha_k+\e_j}=\begin{cases} 1 & \text{ if } \alpha_i= \alpha_k+\e_j\\
0 & \text{ if } \alpha_k+\e_j \in E, \; \alpha_i \neq \alpha_k+\e_j \\
\mu_{\alpha_i, \alpha_k+\e_j} & \text{ if  } \alpha_k+\e_j \not\in E.
\end{cases}
\end{eqnarray}
The {\em parametric multiplication matrices} corresponding to $E$ are defined  
for $i=1, \ldots, n$ by
\begin{equation}\label{Mmu}
\mM_{i}^{t}({\bmu}):= \begin{array}{|ccccc|}
\cline{1-5}
  0&\bmu_{\balpha_1, \e_i}&\bmu_{\balpha_2, \e_i}&\cdots &\bmu_{\balpha_{\mult-1}, \e_i}  \\ 
 0&0&\bmu_{\balpha_2,\balpha_1+\e_i}&\cdots &\bmu_{\balpha_{\mult-1},\balpha_1+\e_i}\\  
 \vdots &\vdots &&&\vdots\\
0&0&0&\cdots &\bmu_{\balpha_{\mult-1},\balpha_{\mult-2}+\e_i}\\ 
0&0&0&\cdots &0\\
\cline{1-5} 
\end{array} ,
\end{equation}
We denote by 
$$
\mM(\bmu)^\bgamma:=\mM_1(\bmu)^{\gamma_1}\cdots\mM_n(\bmu)^{\gamma_n},
$$
and note  that for general parameters values $\bmu$, the matrices  $\mM_i(\mu)$ do not commute, so we fix their order by their indices in the above definition of  $\mM(\mu)^\bgamma$. Later we will introduce equations to  enforce pairwise commutation of the parametric multiplication matrices, see Theorems \ref{theorem1} and \ref{theorem2}.  
\end{definition}

\begin{remark} \label{reduce} Note that we can reduce the number of free parameters in the parametric multiplication matrices by further exploiting the commutation rules of the multiplication matrices corresponding to a given  primal basis $B$. For example, consider the breadth one case, where we can assume that $E=\{{\bf 0}, \e_1, 2\e_1, \ldots, (\delta-1)\e_1\}$. In this case free parameters only appear in the first columns of $\mM_2(\mu), \ldots, \mM_n(\mu)$,  the other columns are shifts of these.   Thus, it is enough to introduce $ (n-1)(\delta-1)$ free parameters,  similarly as in \cite{LiZhi2013}. In Section \ref{Sec:Examples} we present a modification of \cite[Example 3.1]{LiZhi2013} which has breadth two, but also uses at 
most $ (n-1)(\delta-1)$ free parameters.
\end{remark}

\begin{definition}[Parametric normal form] \label{parnorm} Let $\KK\subset \K$ be a field. We define
\begin{eqnarray*}
\Nc_{\bz,\bmu} : \KK[\bx] & \rightarrow & \KK[\bz,\bmu]^{\mult}\\
p& \mapsto& \Nc_{\bz,\bmu} (p) := p(\bz +\mM(\mu))[1]= \sum_{\bgamma\in \N^n}
\frac{1}{\bgamma!} \bpartial_{\bz}^{\bgamma}(p) \, \mM({\bmu})^{\bgamma}[1].
\end{eqnarray*}
where $[1]=[1,0,\ldots,0]$ is the coefficient vector of $1$ in the basis $B$. This sum has bounded degree for all $p$ since for $|\bgamma|> \nil$, 
$\mM({\bmu})^{\bgamma}=0$, so the entries of $\Nc_{\bz,\bmu} (p)$ are polynomials in ${\bmu}$ of degree at most $\nil$.

\end{definition}

Notice that this notation is not ambiguous, assuming that the matrices $\mM_{i}(\mu)$
($i=1,\ldots,n$)  are commuting.
The specialization at $(\bx,\bmu)= (\bxi,\bnu)$ gives the coefficient vector $[p]$ of $N(p)$:
$$ 
\Nc_{\bxi,\bnu} (p) =[\Lambda_{{0}} (p),
\ldots, \Lambda_{{\mult-1}} (p)]^{t}\in \K^{\mult}.
$$

\subsection{The multiplicity structure equations of a singular point}

\noindent We can now characterize the multiplicity structure by
polynomial equations.
\begin{theorem} \label{theorem1}
Let $\KK\subset \C$ be any field, $\fb\in\KK[\bx]^N$, and let $\bxi\in
\C^n$ be an isolated solution of $\fb$.  Let $Q$ be the primary ideal at $\xi$ and assume that $B$ is a basis for $\KK[\bx]/Q$ satisfying the conditions of Lemma \ref{pdlemma}. Let $E\subset\N^n $ be as in (\ref{E}) and $\mM_i(\bmu)$ for $i=1, \ldots n$ be the parametric multiplication matrices corresponding to $E$ as in~(\ref{Mmu})  and~$\Nc_{\bxi,\bmu}$ be the parametric normal form as in Defn.~\ref{parnorm} at $\bz=\bxi$.  
Then the ideal $L_{\bxi}$ of $\C[\bmu]$ generated by the polynomial system 
{%\small
\begin{eqnarray}\label{matrixeq}
\begin{cases} 
\Nc_{\bxi,\bmu} (f_{k})\;\; \text{ for } k=1, \ldots, N, \\
\mM_{i}({\bmu})\cdot \mM_{j}({\bmu})-\mM_{i}({\bmu})\cdot \mM_{i}({\bmu})\;\;\text{ for } i, j=1, \ldots, n
\end{cases}
\end{eqnarray}}
is the maximal ideal
$$
\m_{\nu}= (\bmu_{\balpha, \bbeta}- \nu_{\balpha, \bbeta}, (\balpha,\bbeta)\in E\times \partial(E))
$$ 
where $\nu_{\balpha, \bbeta}$ are the coefficients of the dual basis defined in~(\ref{Macbasis}). 
\end{theorem}
\begin{proof}
 As before, the system \eqref{matrixeq} has a solution
$\bmu_{\balpha,\bbeta}=\nu_{\balpha,\bbeta}$ for $(\balpha,\bbeta)\in
E\times \partial(E)$. Thus $L_{\bxi}\subset \m_{\nu}$.

Conversely, let $C=\K[\mu]/L_{\bxi}$ and consider the map 
$$ 
 \Phi: \K[\bx] \rightarrow C^{\mult}, \;\;
p  \mapsto  \Nc_{\bxi ,\bmu}(p)\!\mod L_\xi.
$$
Let $K$ be its kernel.
%i.e. $\Phi(p)$ is the combination of the elements in $B$ with
%coefficients $\Nc_{\bxi ,\bmu}(p)\in C^\mult$. 
%First we show that $\Phi$ is surjective. 
Since the matrices $\mM_{i} (\bmu)$ are commuting
modulo $L_{\bxi}$, we can see that  $K$ is an ideal. 
As $f_{k}\in K$, we have $I:=\langle f_{1}, \ldots, f_N\rangle \subset K$. 

Next we show that $ Q \subset K$.
By construction, for any $\alpha\in \N^{n}$ we have modulo $L_\xi$
\begin{equation*}\label{eq:rel1}
\Nc_{\bxi ,\bmu}((\bx-\bxi)^{\alpha})= \sum_{\bgamma\in \N^n}
\frac{1}{\bgamma!} \bpartial_{\bxi}^{\bgamma}((\bx-\bxi)^{\alpha}) \, \mM({\bmu})^{\bgamma}[1] =  \mM({\bmu})^{\alpha}[1].
\end{equation*}
Using the previous relation, we check that $\forall p,q \in \K[\bx]$, 
\begin{equation}\label{eq:prod}
\Phi(p q) = p(\bxi+\mM (\bmu)) \Phi(q)
\end{equation}
Let $q\in Q$. As $Q$ is the $\m_{\bxi}$-primary component of $I$,
there exists $p\in \C[\bx]$ such that $p (\bxi)\neq 0$ and $p\, q\in
\I$. By~\eqref{eq:prod},~we~have 
$$ 
\Phi (p\, q)= p(\bxi+\mM (\bmu)) \Phi (q) = 0.
$$
Since $p (\bxi)\neq 0$ and $ p(\bxi+\mM (\bmu))= p (\bxi) Id + N$ with $N$
lower triangular and nilpotent, $p(\bxi+\mM (\bmu))$ is invertible.
We deduce that $\Phi(q)= p (\bxi+\mM (\bmu))^{-1}\Phi (pq) =0$ and $q \in K$.

Let us show now that $\Phi$ is surjective and more precisely, that
$\Phi ((\bx-\bxi)^{\alpha_{k}})=\e_{k}$  for $k=0, 
\ldots, \mult-1$ (abusing the notation, as here
$\e_k$ has length $\mult$ not $n$ and $\e_i$ has a $1$ in position $i+1$). Since $B$ is connected to $1$,
either $\alpha_k=0$ or there exists $\alpha_j\in E$ such that
$\alpha_k=\alpha_j+\e_i$ for some $i\in \{1, \ldots, n\}$. Thus the
$j^{\rm th}$ column of $\mM_i({\bmu})$ is $\e_k$ by (\ref{defmuE}).  As
$\{\mM_{i}({\bmu}): i=1, \ldots, n\}$ are  pairwise commuting, we have
$\mM({\bmu})^{\alpha_k}=\mM_i(\bmu) \mM({\bmu})^{\alpha_j}$, and if we
assume by induction on $|\alpha_{j}|$ that  $\mM({\bmu})^{\alpha_j}[1]=\e_j$, we obtain \mbox{$\mM({\bmu})^{\alpha_k}[1]=\e_k$}. 
Thus, for $k = 0,\dots,\mult-1$,  
$\Phi ((\bx-\bxi)^{\alpha_{k}})=\e_{k}$.

We can now prove that $\m_{\nu}\subset L_{\bxi}$. As $M_{i}(\nu)$ is the multiplication by $(x_{i}-\xi_{i})$
in $\C[\bx]/Q$, for any $b\in B$ and $i=1,\ldots,n$, we have
$(x_{i}-\xi_{i})\, b = M_{i}(\nu) (b) + q$ with $q\in Q\subset K$. We deduce that for $k=0,\ldots,\mult-1$,
$$\Phi ((x_{i}-\xi_{i})\, (\bx -\bxi)^{\alpha_{k}}) =  \mM_{i}(\bmu) \Phi
((\bx -\bxi)^{\alpha_{k}}) = \mM_{i}(\bmu) \e_{k} =
\mM_{i}(\nu) \e_{k}.$$

This shows that $\bmu_{\alpha,\beta}-\nu_{\alpha,\beta}\in L_{\bxi}$
for $(\alpha,\beta) \in E\times\partial(E)$ and that $\m_{\nu}=L_{\bxi}$.
\end{proof}

In the proof of the next theorem  we  need to consider cases when the multiplication matrices do not commute. We introduce the following definition:

\begin{definition}\label{commutator}
Let $\KK\subset \C$ be any  field. Let $\Cc$ be the ideal of $\KK[\bz,\mu]$ generated by entries of the commutation
relations:
$\mM_{i}({\bmu})\cdot \mM_{j}({\bmu})-\mM_{j}({\bmu})\cdot
\mM_{i}({\bmu})=0$, $i,j=1,\ldots,n$. We call $\Cc$ the {\em commutator ideal}.
\end{definition}

\begin{lemma}\label{commlemma} For any field $\KK\subset \C$, $p\in \KK[\bx]$, and $i=1,\ldots,n$, we have
\begin{equation}
\Nc_{\bz,\bmu} ( x_{i} p) = z_{i} \Nc_{\bz, \bmu} (p) + \mM_{i} (\bmu)\, \Nc_{\bz, \bmu} (p) + O_{i, \bmu}(p). \label{eq:nf}
\end{equation}
where $O_{i, \mu}: \KK[\bx]\rightarrow \KK[\bz, \mu]^{\mult}$ is linear with image in the commutator ideal $\Cc$.
\end{lemma} 
%\negskip
\begin{proof}
$\Nc_{\bz, \bmu} (x_{i} p) 
=  
\sum_{\bgamma}
\frac{1}{\bgamma!} \, \partial_{\bz}^{\bgamma} (x_{i} p) \,
  \mM(\bmu)^{\bgamma}[1]
$
  \begin{eqnarray*}
& = &
z_{i} \sum_{\bgamma} \frac{1}{\bgamma!} \partial_{\bz}^{\bgamma} (p) \, \mM({\bmu})^{\bgamma}[1] +
\sum_{\bgamma} 
\frac{1}{\bgamma!} \gamma_{i}\, \partial_{\bz}^{\bgamma-e_{i}} (p) \,
\mM({\mu})^{\bgamma}[1]\\ 
& =&
z_{i} \sum_{\bgamma} \frac{1}{\bgamma!} \partial_{\bz}^{\bgamma} (p)
\, \mM({\bmu})^{\bgamma}[1] 
+ \sum_{\bgamma} \frac{1}{\bgamma!} \partial_{\bz}^{\bgamma} (p) \,
\mM({\bmu})^{\bgamma+e_{i}}[1]\\
& = &
z_{i} \, \Nc_{\bz, \bmu} (p) 
+ \mM_{i}(\bmu) \left ( \sum_{\bgamma} \frac{1}{\bgamma!} \partial_{\bz}^{\bgamma} (p) \,
\mM({\bmu})^{\bgamma}[1] \right) \\
&&~~~~~~+ \sum_{\bgamma} \frac{1}{\bgamma!} \, \partial_{\bz}^{\bgamma} (p) \,
O_{i, \bgamma}({\bmu})[1]  \\
\end{eqnarray*}
where $O_{i, \bgamma}(\mu)= \mM_{i} (\bmu) \mM (\bmu)^{\bgamma}- \mM (\bmu)
^{\bgamma+e_{i}}$ is a $\mult \times \mult$ matrix with coefficients in $\Cc$.
Therefore, $O_{i, \mu}:p\mapsto\sum_{\bgamma} \frac{1}{\bgamma!} \partial_{\bz}^{\bgamma} (p) \,
O_{i,\bgamma}({\bmu})[1]$ is a linear functional of $p$ with coefficients
in $\Cc$.
\end{proof}

The next theorem proves that the system defined as in~(\ref{matrixeq}) for general $\bz$ has $(\bxi, {\nu})$ as a simple root.

\begin{theorem}\label{theorem2}
Let $\KK\subset \C$ be any field, $\fb\in\KK[\bx]^N$, and let $\bxi\in
\C^n$ be an isolated solution of $\fb$.  Let $Q$ be the primary ideal at $\xi$ and assume that $B$ is a basis for $\KK[\bx]/Q$ satisfying the conditions of Lemma \ref{pdlemma}. Let $E\subset\N^n $ be as in (\ref{E}) and $\mM_i(\bmu)$ for $i=1, \ldots n$ be the parametric multiplication matrices corresponding to $E$ as in~(\ref{Mmu})  and~$\Nc_{\bz,\bmu}$ be the parametric normal form as in Defn.~\ref{parnorm}.
Then $(\bz, {\mu})=(\bxi, {\nu})$ is an isolated  root with multiplicity one of the polynomial system in  $\KK[\bz, {\mu}]$:
{%\small
\begin{eqnarray}\label{overdet}
\begin{cases}  
\Nc_{\bz,\bmu} (f_{k}) = 0 \;\text{ for } k=1, \ldots, N,\\
\mM_{i}({\bmu})\cdot \mM_{j}({\bmu})-\mM_{j}({\bmu})\cdot \mM_{i}({\bmu})=0\;
\text{ for } i, j=1, \ldots, n.  \end{cases}
\end{eqnarray}}
%where we use the shorthand $M_{\bx-\bxi}^{{\gamma}}= M^{\gamma_1}_{x_1-\xi_1}\cdots M^{\gamma_n}_{x_n-\xi_n}$. 
\end{theorem}
\begin{proof}
  For simplicity, let us denote the (non-zero) polynomials appearing in (\ref{overdet}) by 
$$P_1, \ldots, P_M\in  \KK[\bz, {\bmu}],$$ 
where $M\leq N\mult+n(n-1)(\mult-1)(\mult-2)/4$. To prove the theorem, it is sufficient to prove that the columns of the Jacobian matrix of the system $[P_1, \ldots, P_M]$ at $(\bz, {\bmu})=(\bxi, {\nu})$ are linearly independent. The columns of this Jacobian matrix correspond to the elements in $\K[\bz, {\bmu}]^*$ 
$$\partial_{1, \xi}, \ldots, \partial_{n, \xi}, \text{ and } \partial_{\mu_{\balpha, \bbeta}}\;\text{ for } \; (\balpha, \bbeta) \in E\times \partial(E),
$$ 
where $\partial_{i, \xi}$ is defined in (\ref{partial}) for  $\bz$ replacing $\bx$, and $\partial_{\bmu_{\balpha, \bbeta}}$ is defined by
$$
\partial_{{\bmu_{\balpha, \bbeta}}}(q) = \frac{d q}{ d \bmu_{\balpha, \bbeta}}\left|_{(\bz, {\mu})=(\bxi, {\nu})} \right. \quad \text{ for } q\in \K[\bz, {\mu}]. 
$$
Suppose there exist $a_1, \ldots, a_n,$ and $a_{\balpha, \bbeta}\in \C$ for  $(\balpha,\bbeta) \in  E\times \partial(E)$ not all zero 
such that 
$$
\Delta:= a_1\partial_{1, \xi}+ \cdots + a_n\partial_{n, \xi}+\sum_{\balpha, \bbeta} a_{\balpha, \bbeta} \partial_{\mu_{\balpha, \bbeta}}\in \K[\bz, {\bmu}]^*
$$
vanishes on all polynomials $P_1, \ldots, P_M$ in (\ref{overdet}).  In
particular, for an element $P_{i} (\bmu)$ corresponding to the commutation
relations and any polynomial $Q  \in \K[\bx, \mu]$, using the product rule for the linear differential operator $\Delta$ we get  
$$
\Delta (P_{i} Q)= \Delta (P_{i}) Q (\bxi,\bnu) + P_{i} (\bnu) \Delta (Q) = 0
$$
since $ \Delta (P_{i}) =0$ and $P_{i} (\bnu)=0$. By the linearity of $\Delta$, for any
polynomial $C$ in the commutator ideal $ \Cc$ defined in Defn. \ref{commutator}, we have $\Delta (C)=0$.

Furthermore, since $\Delta(\Nc_{\bz, \bmu}(f_k))=0$ and by $$\Nc_{\bxi, \bnu} (f_k) = [\Lambda_{0}(f_k), \ldots, \Lambda_{{\mult-1}}(f_k)]^t,$$ we get that 
{\small \begin{equation}\label{dualeq}
 (a_1\partial_{1, \xi}+ \cdots+a_n\partial_{n, \xi})\cdot \Lambda_{{\mult-1}}(f_k)+ \sum_{|\bgamma|\leq |\balpha_{\mult-1}|}p_{\bgamma}({\nu}) \; \bpartial^{\bgamma}_{\xi}(f_k)=0
\end{equation}}
where $p_{\bgamma}\in \K[{\mu}]$ are  some polynomials in the
variables $\mu$ that do not depend on $f_k$. 
If $a_{1}, \ldots, a_{n}$ are not all zero, we have an element $\tilde{\Lambda}$ of $\K[\bpartial_\bxi]$ of order strictly greater than
${\rm ord} (\Lambda_{{\mult-1}})=\nil$ that vanishes on $f_1, \ldots, f_N$. 

Let us prove that this higher order differential also vanishes on all multiples of  $f_k$ for
$k=1, \ldots, N$.
Let $p\in \K[\bx]$ such that $\Nc_{\bxi,\bnu} (p)=0$, $\Delta
(\Nc_{\bz,\bmu} (p))=0$. Since the multiplication matrices commute at $\mu=\nu$, we have by Lemma \ref{commutator}
\begin{eqnarray*}
{\Nc_{\bxi,\bnu} ((x_{i}-\xi_{i}) p)}= 
(x_{i}-\xi_{i}) \Nc_{\bxi,\bnu} (p) + 
\mM_{i} (\nu) \Nc_{\bxi,\bnu} (p)  =  0
\end{eqnarray*}
and  by~\eqref{eq:nf} we have
\begin{eqnarray*}
  {\Delta (\Nc_{\bz, \bmu} ((x_{i}-\xi_{i}) p))}
&= &
\Delta ((x_{i}-\xi_{i}) \Nc_{\bz, \bmu} (p)) + 
\Delta (\mM_{i} (\mu) \Nc_{\bz, \bmu} (p)) + \Delta( O_{\mu} (p))\\
 & = &
\Delta (x_{i}-\xi_{i}) \Nc_{\bxi, \bnu} (p) + (\xi_{i}-\xi_{i})\Delta( \Nc_{\bz, \bmu} (p)) \\
&& ~~~~ + \Delta (\mM_{i} (\mu))  \Nc_{\bxi, \bmu} (p) + 
\mM_{i} (\nu)  \Delta (\Nc_{\bz, \bmu} (p)) \\
&& ~~~~ +  \Delta (O_{i,\bmu} (p))\\
& = & 0.
\end{eqnarray*}
As $\Nc_{\bxi,\bnu} (f_{k})=0$, $\Delta (\Nc_{\bz,\bmu} (f_{k}))=0$,
$i=1,\ldots, N$, we
deduce by induction on the degree of the multipliers and by linearity that for any
element $f$ in the ideal $I$ generated by $f_{1}, \ldots, f_{N}$, we
have 
$$ 
\Nc_{\bxi,\bnu} (f)=0 \hbox{~~~and~~~} \Delta (\Nc_{\bz,\bmu} (f))=0,
$$
which yields $\tilde{\Lambda} \in I^{\bot}$. Thus we have
$\tilde{\Lambda} \in I^{\bot}\cap \K[\bpartial_\bxi]= Q^{\bot}$ (by Lemma
\ref{lem:primcomp}).
As there is no element of degree strictly bigger than $\nil$ in
$Q^{\bot}$, this implies that 
$$a_1=\cdots=a_n=0.$$
Then, by specialization at $\bx=\bxi$, $\Delta$ yields an element of the kernel
of the Jacobian matrix of the system \eqref{matrixeq}.
By Theorem \ref{theorem1}, this Jacobian has a zero-kernel, since it defines
the simple point $\nu$. We deduce that $\Delta=0$ and $(\bxi,\bnu)$ is
an isolated and simple root of the system  \eqref{overdet}.
\end{proof}

The following corollary applies the polynomial system defined in (\ref{overdet}) to refine the precision of an approximate multiple root together with the coefficients of its Macaulay dual basis. The advantage of using this, as opposed to using the Macaulay multiplicity matrix, is that the number of variables is much smaller, as was noted above.  

\begin{corollary} Let $\fb\in \KK[\bx]^N$ and $\bxi\in \C^n$ be as above, and let $\Lambda_{0}({\nu}), \ldots, \Lambda_{{\mult-1}}({\nu})$ be its dual basis as in~(\ref{Macbasis}).  Let $E\subset \N^n$ be as above. Assume that we are given  approximates for the singular roots and its inverse system as in (\ref{Macbasis})
$$
\tilde{\bxi} \cong \bxi \; \text{ and } \; \tilde{\nu}_{\alpha_i, \bbeta}\cong \nu_{\alpha_i, \bbeta} \;\;\forall \balpha_i \in E,\; \beta\not\in E, \;|\bbeta|\leq \nil.
$$
Consider the overdetermined system in $\KK[\bz, \mu]$ from (\ref{overdet}).
Then a square system of  random linear combinations  of the polynomials in (\ref{overdet}) will have 
a simple root at $\bz=\bxi$, $\mu=\nu$ with high probability. Thus,  we can apply Newton's method for this square system to refine $\tilde{\bxi}$ and $\tilde{\nu}_{\alpha_i, \bbeta}$ for $(\balpha_i, \bbeta)\in E\times \partial(E)$. For   $\tilde{\nu}_{\alpha_i, \bgamma}$ with $\bgamma\not \in E^+$ we can use (\ref{restnu}) for the update.  
\end{corollary}

\begin{example}\label{Ex:Illustrative2}
Reconsider the setup from Ex.~\ref{Ex:Illustrative} with primal
basis $\{1,x_2\}$ and $E = \{(0,0),(0,1)\}$.  We obtain
$$\mM_1^{t}(\mu) = \left[\begin{array}{cc} 0 &  \mu \\ 0 & 0 \end{array}\right]~~\hbox{and}~~
\mM_2^{t}(\mu) = \left[\begin{array}{cc} 0 & 1 \\ 0  & 0 \end{array}\right].$$
The resulting deflated system in (\ref{overdet}) is
{\small
$$F(z_1,z_2,\mu) = \left[\begin{array}{c} z_1 + z_2^2 \\
\mu + 2 z_2 \\ z_1^2 + z_2^2 \\ 2 \mu z_1 + 2 z_2 \end{array}\right]$$
}
which has a nonsingular root at $(z_1,z_2,\mu) = (0,0,0)$ corresponding
to the origin with multiplicity structure $\{1,\partial_{2}\}$.
\end{example}
We remark that, even if $E$ does not correspond to an
orthogonal primal-dual basis, it can define an isolated root. 
The deflation system will have an isolated simple solution as soon as
the parametric multiplication matrices are upper-triangular and nilpotent.
This is illustrated in the following example:
\begin{example}
We consider the system: $f_1= x_1 - x_2 + x_1^2$, $f_2= x_1
- x_2 + x_2^2$ of Example \ref{ex:mult3}. The point $(0,0)$ is a root
of multiplicity $3$. We take $B=\{1, x_{1},x_{2}\}$, which does not
correspond to a primal basis of an orthogonal primal-dual pair.
The parametric multiplication matrices are:
$$ 
M_{1}^{t}(\mu) =  \left[ \begin {array}{ccc} 0&1&0\\0&0&\mu_1 \\
                           0&0&0\end {array} \right],\ 
M_{2}^{t}(\mu) =  \left[ \begin {array}{ccc} 0&\mu_{2}&1\\ 0&0&\mu_{{3}}\\ 0&0&0\end {array} \right] 
$$
The extended system is generated by the commutation relations 
 $M_{1}M_{2}-M_{2}M_{1}=0$, which give the polynomial
 $\mu_{{1}}\mu_{{2}}-\mu_{{3}},$
 and the normal form relations:
\begin{itemize}
 \item $\Nc(f_{1})=0$ gives the polynomials $x_{{1}}-x_{{2}}+{x_{{1}}}^{2},\ 1+2\,x_{{1}}-\mu_{{2}},\ -1+\mu_{{1}},$
 \item $\Nc(f_{2})=0$ gives the polynomials  $x_{{1}}-x_{{2}}+{x_{{2}}}^{2},\ 1+ \left( -1+2\,x_{{2}} \right) \mu_{{2
}},\ -1+2\,x_{{2}}+\mu_{{2}}\mu_{{3}}$
\end{itemize}
To illustrate numerically that this extended system in the variables
$(x_{1},x_{2},\mu_{1},\mu_{2},\mu_{3})$ defines a simple
root, we apply Newton iteration on it starting from a point close to
the multiple solution $(0,0)$ and its inverse system:
{\small\begin{center}
\begin{tabular}{ll}
  Iteration & $[x_{1},x_{2},\mu_{1},\mu_{2},\mu_{3}]$ \\
  0 & {$[ 0.1, 0.12, 1.1, 1.25, 1.72]$}\\
  1 & {$[ 0.0297431315,  0.0351989647, 0.9975178694, 1.0480778978, 1.0227973199]$}\\
  2 & {$[ 0.0005578682,  0.0008806394, 0.9999134370, 0.9997438194, 0.9996904740]$}\\
  3 & {$[ 0.0000001981,-0.0000001864, 0.9999999998, 1.0000002375, 1.0000002150]$}\\
  4 & {$[{ 2.084095775\, 10^{-14}},-{ 1.9808984139\, 10^{-14}},  1.0, 1.0000000000, 1.0000000000]$}
\end{tabular}
\end{center}}
As expected, we observe the quadratic convergence to the simple
solution $(\bxi,\nu)$ corresponding to the point $(0,0)$ and the
dual basis
$$\left\{1, \partial_{1}+\nu_{2}\partial_{2}, {\partial_2} + {1\over
    2}\nu_{1}\partial_1^2+\nu_{3}\partial_1\partial_2+{1\over
    2}\nu_{2}\nu_{3}\partial_2^2\right\}$$ 
with $\nu_{1}=1,\nu_{2}=1,\nu_{3}=1$.
\end{example}

\subsection{Deflation ideals}\label{DeflIdeal}

In this section we study a similar approach as in \cite{Leykin2008}, where a so called {\em deflation ideal} $I^{(d)}$ is defined for an arbitrary ideal $I\subset\KK[\bx]$ and $d\geq 0$. Here we define a modification of the construction of \cite{Leykin2008}, based on our construction in Theorem \ref{theorem2}, which we call the {\em $E$-deflation ideal}.

\begin{definition} \label{E-defl}Let $\fb=(f_1, \ldots, f_N)\in \KK[\bx]^N$ and $I=\langle  f_1, \ldots, f_N\rangle$. Let 
  $$
  E=\{\alpha_0, \ldots, \alpha_{\delta-1}\}\subset\N^n
$$
be a set of $\delta$ exponent vectors  {\em stable under subtraction}, i.e., if $\alpha, \beta\in \N^n$ and $\beta\leq \alpha$ componentwise,  then $\alpha\in E$ implies $\beta\in E$.  We also assume that $\alpha_0=0$ and 
$$|\alpha_0|\leq\cdots \leq |\alpha_{\delta-1}|.
$$Let 
$$\mu:=(\mu_{\alpha_i, \alpha_k+\e_j}\;:\; \alpha_i, \alpha_k\in E, j=1, \ldots, n,|\alpha_i|\geq |\alpha_k|+1, \; \alpha_k+\e_j\not\in E)
$$
be new indeterminates of cardinality $D\leq n\delta(\delta-1)/2$. Let $\mM_i(\mu)$ for $i=1, \ldots, n$  be the parametric multiplication matrices corresponding to $E$ defined in (\ref{Mmu}). Then we define the {\em $E$-deflated ideal} $I^{(E)}\subset\KK[\bx,\mu]$ as
$$I^{(E)}:=\left( \Nc_{\bx,\bmu} (f_{k})  \;:\; k=1, \ldots, N\right)+\left( 
\mM_{i}({\bmu})\cdot \mM_{j}({\bmu})-\mM_{j}({\bmu})\cdot \mM_{i}({\bmu})\;
:\; i, j=1, \ldots, n\right).
$$
Here $ \Nc_{\bx,\bmu}$ is the parametric normal form defined in Defn. \ref{parnorm} for $\bz=\bx$.
\end{definition}

First we prove that the $E$-deflation ideal  does not depend on the choice of the generators of $I$.

\begin{proposition}
Let $I\subset \KK[\bx]$ and $E\subset\N^n$ be as above. Then, the $E$-deflation ideal $I^{(E)}$ does not depend on the generators $f_1, \ldots , f_N$ of $I$.
\end{proposition}

\begin{proof}
By Lemma \ref{commlemma},  we have
\begin{equation*}
\Nc_{\bx,\bmu} ( x_{i} p) = x_{i} \Nc_{\bx, \bmu} (p) + \mM_{i} (\bmu)\, \Nc_{\bx, \bmu} (p) + O_{i, \bmu}(p),
\end{equation*}
where $O_{i, \bmu}(p)$ is a vector of polynomials in the commutator ideal $\Cc$ as in Defn.~\ref{commutator}. Thus, if $ \Nc_{\bx, \bmu} (p) \in I^{(E)}$ then $\Nc_{\bx,\bmu} ( x_{i} p)\in I^{(E)}$. Using induction on the degree of $\bx^\alpha$, we can  show that $ \Nc_{\bx, \bmu} (p) \in I^{(E)}$ implies that  $\Nc_{\bx,\bmu} ( \bx^\alpha p)\in I^{(E)}$. Using that  $\Nc_{\bx,\bmu}$ is linear, we get $\Nc_{\bx,\bmu} ( I)\subset I^{(E)}$. 
\end{proof}

Next, we prove the converse of Theorem \ref{theorem2}, namely that
isolated simple roots of $I^{(E)}$ correspond to  multiple
roots of $I$ with multiplicity structure corresponding to $E$, at least up to the order of $E$. 

\begin{theorem}\label{theorem3}
Let $I=\langle f_1, \ldots, f_N\rangle \subset\KK[\bx]$ and
$E=\{\alpha_0, \ldots, \alpha_{\delta-1}\}\subset\N^n$ be as in
Definition \ref{E-defl} and let $o= |\alpha_{\delta-1}|$. Let $(\xi,
\nu)\in \C^{n+D}$ be an isolated solution of the $E$-deflated
ideal $I^{(E)}\subset \KK[\bx,\mu]$. Then $\xi$ is a root of $I$, and
$(\xi, \nu)$ uniquely determines an orthogonal pair of primal-dual
bases $B$ and $\bLambda$.
They satisfy the conditions of Lemma \ref{pdlemma}
for $\C[x]/Q$ and its dual, respectively,
where $Q=  I_{\xi} + \m_{\xi}^{o+1}$ with $I_{\xi}$ the
intersection of the primary components of $I$ contained in $\m_{\xi}$.
\end{theorem}

\begin{proof} 
Since $ \Nc_{\bx,\bmu} (f_{k})[1] =f_k$, we have  $f_1, \ldots, f_N\in
I^{(E)}$,  thus  $\xi\in V(I)$.
The monomial set $B=\{(\bx-\bxi)^{\alpha_i}\;:\; i=0, \ldots,
\delta-1\}$ is stable by derivation and thus connected to $1$ (i.e. if $m\in \xb^{E}$
and $m\neq 1$, there exists $m'\in \xb^{E}$ and $i\in [1,n]$ such
that $m=x_{i} m'$).
The matrices $\{\mM_i(\nu)\}$ associated to the rewriting  family
$${\mathcal F}:=\left\{(\bx-\bxi)^{\alpha_k+\e_j}-\sum_{i< k} \nu_{\alpha_i, \alpha_k+\e_j} (\bx-\bxi)^{\alpha_i}\; :\; \alpha_k\in E, \;\alpha_k+\e_j\not\in E\right\}
$$
are pairwise commuting. By \cite{Mourrain99-nf,
  MourrainTrebuchet2005}, ${\mathcal F}$ is a border basis for
$B$
and $B$ is a basis of $\C[x]/Q$ where 
 $Q:=( {\mathcal F})\subset \C[\bx]$.  In particular, $\dim \C[x]/Q=\delta$.
Since the matrices $\mM_i(\nu)$ are strictly lower triangular, the
elements of $\C[x]/Q$ are nilpotent, so $Q$ is a $\m_{\xi}$-primary
ideal.
By Lemma \ref{lem:normalform} the dual basis $\bLambda=(\Lambda_0, \ldots,
\Lambda_{\delta-1})$ is
\begin{eqnarray*}
&&\Lambda_i := \sum_{\gamma\in \N^n}[\mM(\nu)^{\bgamma}]_{i,1}\frac{1}{\gamma!}\partial_\xi^\gamma, \text{  using the identity }\\
&&\nu_{\balpha_i,\bgamma}= [\mM(\nu)^{\bgamma}]_{i,1} \text{ for all } \gamma\in \N^n, \; i=0, \ldots, \delta-1
\end{eqnarray*}
similarly as in (\ref{Macbasis}) and  (\ref{restnu}). By induction on the degree, we prove that for $|\gamma|>|\alpha_i|$ we have $[\mM(\nu)^{\bgamma}]_{i,1}=0$. Thus, $B$ and  $\bLambda$ satisfies the properties of Lemma~\ref{pdlemma}.

Let $\DDD:={\rm span}(\bLambda)$. Then $\DDD$ is stable under derivation since 
\begin{eqnarray*}
\d_{\partial_{j,\xi}}(\Lambda_i)&= &\d_{\partial_{i,\xi}}(\sum_{\gamma\in \N^n}[\mM(\nu)^{\bgamma}]_{i,1}\frac{1}{\gamma!}\partial_\xi^\gamma)=\sum_{\beta\in \N^n}[\mM_j(\nu)\mM(\nu)^{\bbeta}]_{i,1}\frac{1}{\beta!}\partial_\xi^\beta\\
&=&[\mM_j(\nu)]_{i,*}\cdot\left(\sum_{\beta\in \N^n}[\mM(\nu)^{\bbeta}]_{*,1}\frac{1}{\beta!}\partial_\xi^\beta\right)
=\sum_{k=0}^{i-1} [\mM_j(\nu)]_{i,k}\Lambda_k.
\end{eqnarray*}
This implies that $\DDD\subseteq Q^{\perp}$, and comparing dimensions we get equality, i.e.,
$$
q\in Q \quad \Leftrightarrow \quad \Lambda_i(q)=0\text{ for all } i=0, \ldots, \delta-1.
$$
Since $\Lambda_i(f_k)=\Nc_{\bxi,\bnu} (f_{k})[ i]=0$ for all $k=1,
\ldots, N$ and $i=0, \ldots, \delta-1$,~\mbox{$I\subset Q$}.  

Finally, we prove that $Q= I_{\xi}+ \m_{\xi}^{o+1}$.
As $\DDD$ is generated by elements of order~$\leq o$,
$\m_{\xi}^{o+1}\subset Q$. Thus, $I+ \m_{\xi}^{o+1} \subset  
Q$. Localizing at $\m_{\xi}$ yields $I_{\xi} + \m_{\xi}^{o+1} \subset Q$.

We prove now the reverse inclusion: $Q \subset I_{\xi}+ \m_{\xi}^{o+1}$.
Let $\DDD_{\xi}=I_{\xi}^{\perp}\subset \C[\bpartial_{\xi}]$.
Suppose that there exists an element of $\DDD_{\xi}$ of order $\leq
o$, which is not in $\DDD= Q^{\perp}$. Let $\Lambda$ be such a non-zero element of
$\DDD_{\xi}\setminus \DDD$ of smallest possible order $t \leq o$. 
As~$\xi \in V(I)$, we can assume that $t>0$. 
We are going to prove that $(\xi, \nu)$ is not an isolated solution.

By reduction by the basis $\Lambda_{i}$ of
$\DDD$, we can assume that the coefficients of
$\partial^{\alpha_i}_\xi$ are zero in $\Lambda$.  
Thus, for any parameter value $c \in\C$ we can replace $\bLambda$ by  
$$
\bLambda_c:=(\Lambda_0, \ldots,\Lambda_{\delta-1} +c\cdot \Lambda)
$$
so that $B$ and $\bLambda_c$ form a primal-dual pair.

As $t$ is minimal, we have $d_{\partial_{i,\xi}}(\Lambda) \in \DDD_{t-1}$ for all $i\in [1,n]$.  
Thus, there exist coefficients $\nu'_{i,j}$ such that
$$ 
\d_{\partial_{i,\xi}}(\Lambda) = \sum_{i} \nu'_{i,j} \Lambda_{j}.
$$
As $\Lambda$ is of order $t \leq  {\rm ord}(\Lambda_{\delta-1})$ and 
$d_{\partial_{i,\xi}}(\Lambda)$ is of order $<t$, the coefficients
$\nu'_{i,\delta-1}$ must vanish.
This shows that the matrix
$M_{i}^{t}(\nu')=(\bLambda_{c,,j}((\bx-\bxi)^{\alpha_{k}+\e_{i}}))$ is
a nilpotent upper triangular matrix of the form \eqref{Mmu}.

All the coefficients $\nu'_{i,j}$ cannot vanish otherwise $\Lambda$ is
a constant, which is excluded since $t>0$.   Thus for all $c\neq 0$,
the matrices $M_{i}(\nu')^{t}$ 
representing the operators $d_{\partial_{i,\xi}}$ in
the dual basis $\bLambda_{c}$, are distinct from 
$M_{i}(\nu)^{t}$.
These matrices are commuting, since the derivations
$d_{\partial_{i,\xi}}$ commute.
Moreover, for any $\alpha\in \N^{n}$, we have
$$ 
\bLambda_{c} ((\bx-\bxi)^{\alpha}) =
(\bx-\bxi)^{\alpha} \cdot\bLambda_{c} (1)
=  \langle \mM^{t}({\bnu'})^{\alpha}[\bLambda_{c}], [1]\rangle
=  \langle [\bLambda_{c}],  \mM({\bnu'})^{\balpha}[1]\rangle.
$$
As $\Lambda(\fb)=0$, we deduce that 
$\bLambda_{c}(\fb)=0$, $\fb(\bxi+\mM({\bnu'}))[1]=0$ and $\Nc_{\xi,\bnu'}(\fb)=0$.

Therefore, the solution set of the system $I^{(E)}$ contains
$(\xi,\nu')$ for all $c\neq 0$, that is the line through the points $(\xi,\nu)$,  $(\xi,\nu')$, which
implies that $(\xi,\nu)$ is not isolated.
We deduce that if $(\xi,\nu)$ is isolated, then
$\DDD_{\xi,o} \subset \DDD_{o}=\DDD$, that is
$$ 
(I_{\xi}+ \m_{\xi}^{o+1})^{\perp}\subset Q^{\perp}
$$
or equivalently, $Q\subset I_{\xi}+ \m_{\xi}^{o+1}$.
\end{proof}

This theorem implies, in particular, that if $\bxi$ is an isolated root
of $I$ and $o$ is its order, then $Q=I_{\bxi}$ is the
primary component of $I$ associated to $\bxi$.

The following example illustrates that if $\bxi$ is not an isolated root of $I$, but an embedded point, 
then the primary ideal $Q$ in Theorem~\ref{theorem3} may differ from the primary ideal in the primary decomposition of $I$ corresponding to $\xi$.

\begin{example} We consider the ideal $I=\left(x^{2},xy\right)$ with primary decomposition $I= \left( x\right)\cap
 \left(  x^{2},y\right)$, which corresponds to a simple line
$V( x)$ with an embedded point $V(x^{2},y)$ of multiplicity $2$ at $\xi=(0,0)$. With $E_0:=\{(0,0),(1,0)\}$ corresponding to the primal basis $\{1, x\}$, we get parametric multiplication matrices
  $$
  M^{t}_{x}=\left(
  \begin{array}{cc}
    0 & 1 \\ 
    0 & 0  
  \end{array}
\right),~~
M^{t}_{y} =
\left(
  \begin{array}{ccc}
    0 & \mu \\ 
    0 & 0   \end{array}
  \right)
  $$ which are commuting. The parametric normal form  is $$f \mapsto \Nc(f)=[f(\xb), \partial_{x}(f)(\xb)+\mu\partial_{y}(f)(\xb)], $$ so the $E_0$-deflated ideal is   
  $I^{(E_0)}=\left( x^{2}, 2x, xy, y+\mu x\right)=\left( x,y\right)\subset\C[x,y,\mu]$, but $(0,0)$ corresponds to a positive dimensional component  $\{ (0,0,\mu):\mu\in \C \}$ of $I^{(E_0)}$.  \\
   For $E_1=\{(0,0),(1,0),(0,1)\}$ corresponding to the primal basis $\{1, x, y \}$, 
the parametric multiplication matrices are constant
and obviously commute:
$$ 
M^{t}_{x} =
\left(
  \begin{array}{ccc}
    0 & 1 & 0\\ 
    0 & 0 & 0\\ 
    0 & 0 & 0\\ 
  \end{array}
\right),~~
M^{t}_{y} =
\left(
  \begin{array}{ccc}
    0 & 0 & 1\\ 
    0 & 0 & 0\\ 
    0 & 0 & 0\\ 
  \end{array}
  \right).
$$
The parametric normal form is
$f \mapsto \Nc(f)=[f(\xb), \partial_{x}(f)(\xb), \partial_{y}(f)(\xb)]$.

The $E_1$-deflated  ideal $I^{(E_1)}=\left( x^{2}, x, 0, xy, y, x\right)=\left( x,y\right) \subset\C[x,y]$. 
It defines the (smooth) isolated point $\bxi = (0,0)$ and the
associated $\left( x,y\right)$-primary ideal is 
$$ 
Q=\langle 1, \partial_{x}, \partial_{y}\rangle^{\perp}=\left( x^{2},x y,
y^{2}\right) =  I+ \left( x,y\right)^{2} = \left( x,y\right)^{2}\neq  \left( x^{2},y\right).
$$
Similarly, if $E_k:=\{(0,0), (1,0), (0,1), \ldots,(0,k)\}$ corresponding to the primal basis $\{1,x,y,\ldots, y^k\}$, we get that $V(I^{(E_k)})$ is an isolated  simple point with projection $(0,0)$, and the corresponding primary ideal is
$$Q=\langle 1, \partial_{x}, \partial_{y}, \ldots, \partial_y^k\rangle^{\perp}=I+ \left( x,y\right)^{k}=\left( x^2,y\right)  \cap  \left( x,y^{k+1} \right)  \neq  \left(x^{2},y\right).$$
\end{example}
 
\section{Examples}\label{Sec:Examples}
Computations for the following examples, as well as several other 
systems, along with \textsc{Matlab} code can be found at
\url{www.nd.edu/~jhauenst/deflation/}.

\subsection{Caprasse system}\label{Sec:Caprasse}
\noindent We first consider the Caprasse system \cite{Caprasse88,Posso98}:
{\scriptsize
$$
\begin{array}{l}
f(x_1,x_2,x_3,x_4) = 
\left[
\begin{array}{l} 
{x_{{1}}}^{3}x_{{3}}-4\,x_{{1}}{x_{{2}}}^{2}x_{{3}}-4\,{x_{{1}}}^{2}x_{{2}}x_{{4}}-2\,{x_{{2}}}^{3}x_{{4}}-4\,{x_{{1}}}^{2}+10\,{x_{{2}}}^{2}- 4\,x_{{1}}x_{{3}}+10\,x_{{2}}x_{{4}}-2\\
x_{{1}}{x_{{3}}}^{3}-4\,x_{{2}}{x_{{3}}}^{2}x_{{4}}-4\,x_{{1}}x_{{3}}{x_{{4}}}^{2}-2\,x_{{2}}{x_{{4}}}^{3}-4\,x_{{1}}x_{{3}}+10\,x_{{2}}x_{{4}}- 4\,{x_{{3}}}^{2}+10\,{x_{{4}}}^{2}-2\\
{x_{{2}}}^{2}x_{{3}}+2\,x_{{1}}x_{{2}}x_{{4}}-2\,x_{{1}}-x_{{3}}\\
{x_{{4}}}^{2}x_{{1}}+2\,x_{{2}}x_{{3}}x_{{4}}-2\,x_{{3}}-x_{{1}}
\end{array}\right].
\end{array}
$$
}
The following is a multiplicity $4$ root:
$$
\bxi=(\xi_1, \ldots, \xi_4)=\left( -\frac{2\cdot i}{\sqrt{3}},  -\frac{ i}{\sqrt{3}},  \frac{2\cdot i}{\sqrt{3}} ,\frac{ i}{\sqrt{3}}\right)\in \C^4
$$
of multiplicity 4.

We analyze first the methods for deflating the root $\bxi$.
Using the approaches of \cite{DayZen2005,HauWam13,lvz06}, one iteration suffices.
For example, using an extrinsic and intrinsic version of \cite{DayZen2005,lvz06},
the resulting system consists of 10 and 8 polynomials, respectively,
and 8 and 6 variables, respectively.
Following \cite{HauWam13}, using all minors results in a system
of 20 polynomials in 4 variables which can be reduced to
a system of 8 polynomials in 4 variables using the $3\times3$ minors
containing a full rank~$2\times2$~submatrix.
The approach of \S~\ref{Sec:Deflation} using an $|\bm i|=1$
step creates a deflated system consisting 
of $6$ polynomials in $4$ variables.  
In fact, since the null space of the Jacobian at the root
is $2$ dimensional, adding two polynomials is necessary and sufficient.

We illustrate now the second method, for computing the
multiplicity structure.
The primal basis of  $\bxi$ is given by 
$$B=\{1, x_1-\xi_1, x_2-\xi_2, (x_1-\xi_1)^2\},\quad  \mathrm {with}\ 
E=\{(0,0),\,(1,0),\,(0,1),\,(2,0)\},
$$
 and its orthogonal dual basis has the following structure. 
{\small\begin{eqnarray*}
\Lambda_0&=&1,\\
\Lambda_1&=&\partial_{x_1}+\nu_{x_1, x_3}\partial_{x_3} + \nu_{x_1, x_4}\partial_{x_4},\\
\Lambda_2&=&\partial_{x_2}+\nu_{x_2, x_3}\partial_{x_3} + \nu_{x_2, x_4}\partial_{x_4},\\
\Lambda_3&=&\partial_{x_1^2}/2+ \nu_{x_1^2, x_3}\partial_{x_3}+\nu_{x_1^2, x_4}\partial_{x_4}+\nu_{x_1^2, x_1x_2}\partial_{x_1x_2}
\\&+&\nu_{x_1^2, x_1x_3}\partial_{x_1x_3}+ \nu_{x_1^2, x_1x_4}\partial_{x_1x_4}+\nu_{x_1^2, x_2^2}\partial_{x^2_2}/2\\
&+&\nu_{x_1^2, x_2x_3}\partial_{x_2x_3}+\nu_{x_1^2, x_2x_4}\partial_{x_2x_4}+\nu_{x_1^2, x_3^2}\partial_{x_3^2}/2\\
&+&\nu_{x_1^2, x_3x_4}\partial_{x_3x_4}+\nu_{x_1^2, x_4^2}\partial_{x_4^2}/2.
\end{eqnarray*}
Computing the kernel of the  Macaulay multiplicity matrix 
$$
{\rm Mac}_{d}(\fb, \bxi):=\left[\bpartial_\bxi^\balpha\left(\bx^{\bbeta}f_i(\bx)\right)\right]_{|\bbeta|<d, 1\leq i\leq N,|\balpha|\leq d}.
$$
for $d=2$ (of size $20\times 15$), we get the unique solution 
{\small\begin{eqnarray}\label{nusol}
&&\nu_{x_1, x_3}=-1, \nu_{x_1, x_4}=0, \nu_{x_2, x_3}=1, \nu_{x_2, x_4}=1, \nonumber
\\&&\nu_{x_1^2, x_3}=\frac{\sqrt{3}\cdot i}{8}, 
\nu_{x_1^2, x_4}=\frac{\sqrt{3}\cdot i}{4}, \nu_{x_1^2, x_1x_2}=-\frac{1}{4},\\&& \nu_{x_1^2, x_1x_3}=- \frac{5}{4}, \nu_{x_1^2, x_1x_4}=-\frac{1}{4}, \nu_{x_1^2, x_2^2}=-\frac{1}{2},\nu_{x_1^2, x_2x_3}=-\frac{1}{4},\nonumber
 \\&&\nu_{x_1^2, x_2x_4}=-\frac{1}{2}, \nu_{x_1^2, x_3^2}=1, \nu_{x_1^2, x_3x_4}=-\frac{1}{4}, \nu_{x_1^2, x_4^2}=-\frac{1}{2}.\nonumber
\end{eqnarray}}}

The system of  parametric multiplication matrices corresponding to  $E$ is given by
{\scriptsize\begin{eqnarray*}
&{\mM}_{1}(\mu)^t\, = \, \left[ \begin {array}{cccc} 0&1&0&0\\ \noalign{\medskip}0&0&0&1\\ \noalign{\medskip}0&0&0&\mu_{{x_1^2, x_1x_2}}\\ \noalign{\medskip}0&0&0&0\end {array} \right] ,\quad 
 & {\mM}_{{2}}(\mu)^t\, = \, \left[ \begin {array}{cccc} 0&0&1&0\\ \noalign{\medskip}0&0&0&\mu_{{x_1^2, x_1x_2}}\\ \noalign{\medskip}0&0&0&\mu_{{x_1^2, x^2_2}}\\ \noalign{\medskip}0&0&0&0\end {array} \right],\\
&{\mM}_{{3}}(\mu)^t\, = \, \left[ \begin {array}{cccc} 0&\mu_{{x_1,x_3}}&\mu_{{x_2,x_3}}&\mu_{{x_1^2, x_3}}\\ \noalign{\medskip}0&0&0&\mu_{{x_1^2, x_1x_3}}\\ \noalign{\medskip}0&0&0&\mu_{{x_1^2,x_2x_3}}\\ \noalign{\medskip}0&0&0&0\end {array} \right] , \quad
&
 {\it \mM}_{{4}}(\mu)^t\, = \, \left[ \begin {array}{cccc} 0&\mu_{{x_1,x_4}}&\mu_{{x_2,x_4}}&\mu_{{x_1^2, x_4}}\\ \noalign{\medskip}0&0&0&\mu_{{x_1^2, x_1x_4}}\\ \noalign{\medskip}0&0&0&\mu_{{x_1^2, x_2x_4}}\\ \noalign{\medskip}0&0&0&0\end {array} \right] .
\end{eqnarray*}}
Note that $\mu_{x_1^2, x_3^2},\, \mu_{x_1^2, x_3x_4},\, \mu_{x_1^2, x_4^2}$ do not appear in these multiplication matrices. Each of these matrices are nilpotent, and one can check that the maximal non-zero products of them have  degree 2. To obtain the polynomial system in (\ref{overdet}), we first have to compute 
$$\Nc_{\bx, \bmu}(f_i)=  \sum_{\bgamma}
\frac{1}{\bgamma!} \bpartial_{\bx}^{\bgamma}(f_i) \, \mM({\mu})^{\bgamma}[1]\in \Q[\bx, \mu]^4.$$ 
Note that $\Nc_{\bx, \bmu}(f_i)[1]=f_i$ since the only time the $[1,1]$ entry in $\mM(\mu)^\gamma$ is not zero is when $\gamma=0$. The other entries of $\Nc_{\bx, \bmu}(f_i)$ depend on the $\mu$ variables, 
for example 
{\small\begin{eqnarray*}
\Nc_{\bx, \bmu}(f_1)[4]&=&\left( {x_{{1}}}^{3}-4\,x_{{1}}{x_{{2}}}^{2}-4\,x_{{1}} \right) \mu_{{x_1^2, x_3}}+ \left( -4\,{x_{{1}}}^{2}x_{{2}}-2\,{x_{{2}}}^{3}+10\,x_{{2}} \right) \mu_{{x_1^2, x_4}}\\
&&+3\,x_{{1}}x_{{3}}-4\,x_{{2}}x_{{4}}-4
+\left( -8\,x_{{1}}x_{{4}}-8\,x_{{2}}x_{{3}} \right) \mu_{{x_1^2, x_1x_2}} \\&&
+ \left( 3/2\,{x_{{1}}}^{2}-2\,{x_{{2}}}^{2}-2 \right) \mu_{{x_1,x_3}}+ \left( 3/2\,{x_{{1}}}^{2}-2\,{x_{{2}}}^{2}-2 \right) \mu_{{x_2,x_3}} \mu_{{x_1^2, x_1x_2}}
\\&&-4\,x_{{1}}x_{{2}}\mu_{{x_1,x_4}}-4\,x_{{1}}x_{{2}}\mu_{{x_2,x_4}} \mu_{{x_1^2, x_1x_2}}+ \left( -4\,x_{{1}}x_{{3}}-6\,x_{{2}}x_{{4}}+10 \right)  \mu_{{x_1^2,x_2^2}}\\&&
-4\,x_{{1}}x_{{2}}\mu_{{x_1,x_3}} \mu_{{x_1^2, x_1x_2}}
-4\,x_{{1}}x_{{2}}\mu_{{x_2,x_3}} \mu_{{x_1^2,x_2^2}}\\&&
+ \left( -2\,{x_{{1}}}^{2}-3\,{x_{{2}}}^{2}+5 \right) \mu_{{x_1,x_4}}\mu_{{x_1^2, x_1x_2}}\\&&+ \left( -2\,{x_{{1}}}^{2}-3\,{x_{{2}}}^{2}+5 \right) \mu_{{x_2,x_4}}\mu_{{x_1^2, x_2^2}}
\\&&+ \left( 3/2\,{x_{{1}}}^{2}
-2\,{x_{{2}}}^{2}-2 \right) \mu_{{x_1^2, x_1x_3}}-4\,x_{{1}}x_{{2}}\mu_{{x_1^2, x_2x_3}}-4\,x_{{1}}x_{{2}}\mu_{{x_1^2, x_1x_4}}\\&&+ \left( -2\,{x_{{1}}}^{2}-3\,{x_{{2}}}^{2}+5 \right) \mu_{{x_1^2, x_2x_4}}.\\
\end{eqnarray*}}
 Note that this polynomial is clearly not equal to $\Lambda_3(\bx^\alpha f_1)$ for any $\alpha$, which would be linear in the $\mu$ variables.
 
 The commutator relations appearing in (\ref{overdet}) contain polynomials such as 
 $$
 \mu_{x_1^2, x_2x_3}- \mu_{x_1, x_3} \mu_{x_1^2, x_1x_2}+ \mu_{x_2, x_3} \mu_{x_1^2, x_2^2}, 
 $$
 which is    the only non-zero entry in $\mM_{2}\mM_{3}-\mM_{3}\mM_{2}$. 
 
 Using an elimination order, we computed the following Gr\"obner basis for the $E$-deflated ideal $I^{(E)}$ generated by the polynomials in (\ref{overdet}):
{\small 
\begin{eqnarray*}
&&3\,{x_{{4}}}^{2}+1,\;3\,{x_{{3}}}^{2}+4,\;x_{{4}}+x_{{2}},x_{{3}}+x_{{1}},
\mu_{{x_1,x_3}}+1,\;\mu_{{x_1,x_4}},\; \\
&&
\mu_{{x_2,x_4}}-1,\; \,2\mu_{{x_2,x_3}}+3\,x_{{3}}x_{{4}},
2\,\mu_{{x_1^2,x_2x_4}}+1,\; 8\,\mu_{{x_1^2,x_1x_4}}-3\,x_{{3}}x_{{4}},\;4\,\mu_{{x_1^2,x_4}}-3\,x_{{4}},\; \\
&&
8\,\mu_{{x_1^2,x_2x_3}}-3\,x_{{3}}x_{{4}},\;4\,\mu_{{x_1^2,x_1x_3}}+5,\;16\,\mu_{{x_1^2,x_3}}-3\,x_{{3}},\;2\,\mu_{{x_1^2,x_2^2}}+1,\;8\,\mu_{{x_1^2,x_1x_2}}-3\,x_{{3}}x_{{4}}.
\end{eqnarray*}}
 At $\bx = \bxi=\left( -\frac{2\cdot i}{\sqrt{3}},  -\frac{ i}{\sqrt{3}},  \frac{2\cdot i}{\sqrt{3}} ,\frac{ i}{\sqrt{3}}\right)$ this gives the same solution  $\mu=\nu$ as in (\ref{nusol}).
%\end{example}}

 \subsection{A family of examples}
\noindent In this section, we consider a modification of \cite[Example
3.1]{LiZhi2013}, defining multiple points with breadth $2$. For any $n\geq 2$, the following system has $n$ polynomials, each of degree at most $3$, in $n$ variables:
\begin{eqnarray*} 
x_1^3+x_1^2-x_2^2, \;x_2^3+x_2^2-x_3, \ldots, x_{n-1}^3+x_{n-1}^2-x_n, \;x_n^2.
\end{eqnarray*}
The origin is a multiplicity $\delta:=2^n$ root having breadth $2$ (i.e., the 
corank of Jacobian at the origin is $2$). 

We apply our parametric normal form method described in \S~\ref{Sec:PointMult}. Similarly as in Remark \ref{reduce}, we can reduce the number of free parameters to be at most $(n-1)(\delta-1)$ using the structure of the primal basis $B=\{x_1^ax_2^b:a<2^{n-1}, \;  b<2\}$.

The following table shows the multiplicity, number of
variables and polynomials in the deflated system, and the time (in
seconds) it took to compute this system (on a iMac, 3.4 GHz Intel Core i7 processor, 8GB 1600Mhz DDR3 memory).
Note that when comparing our method to an approach using the
null spaces of Macaulay multiplicity matrices  (see for example \cite{DayZen2005,lvz08}), we found  that for $n\geq 4$ the deflated system derived from the Macaulay multiplicity matrix was too large to compute. This is because  the nil-index at the origin is $2^{n-1}$, so the size of the Macaulay multiplicity matrix is  $\;n\cdot{{2^{n-1}+n-1}\choose{n-1}}\times{{2^{n-1}+n}\choose{n}}$.  
{%\scriptsize
$$\begin{array}{|c|c|c|c|c|c|c|c|}
\hline
\multicolumn{2}{|c|}{} &\multicolumn{3}{|c|}{\hbox{New approach}} & \multicolumn{3}{|c|}{\hbox{Null space}} \\
\hline
n & \hbox{mult} & \hbox{vars} & \hbox{poly} & \hbox{time} & \hbox{vars} & \hbox{poly} & \hbox{time}\\
\hline
2 & 4 & 5 & 9 & 1.476 &8&17&2.157\\
\hline
3 & 8 & 17 & 31 & 5.596&192&241&208 \\
\hline
4 & 16 & 49 & 100 & 19.698 &7189 &19804&>76000\\
\hline
5 & 32 & 129 & 296 & 73.168&N/A&N/A&N/A \\
\hline
6 & 64 & 321 & 819 & 659.59 &N/A&N/A&N/A\\
\hline
\end{array}$$}
% New approach:
% # Poly = 5*d^2 - 6*d + 9
% # Vars = 3*d^2 - 2*d + 4
% Null space: 
% # Poly = 3 + 3*nchoosek(2+d,3)*(2*d-1)
% # Vars = 3 + (2*d-1)*(nchoosek(3+d,3)-2*d)

\subsection{Examples with multiple iterations}\label{Sec:MultipleIterations}
\noindent In our last set of examples, we consider simply deflating a root of the last three systems 
from \cite[\S~7]{DayZen2005}
and a system from \cite[\S~1]{Lecerf02}, each of which 
required more than one iteration to deflate.  
These four systems and corresponding points are:
{\small\begin{itemize}
\item[1:] $\{x_1^4 - x_2 x_3 x_4, x_2^4 - x_1 x_3 x_4, x_3^4 - x_1 x_2 x_4, x_4^4 - x_1 x_2 x_3\}$ at $(0,0,0,0)$ with $\mult = 131$ and $o = 10$;
\item[2:] $\{x^4, x^2 y + y^4, z + z^2 - 7x^3 - 8x^2\}$ at $(0,0,-1)$ with $\mult = 16$ and $o = 7$;
\item[3:] $\{14x + 33y - 3\sqrt{5}(x^2 + 4xy + 4y^2 + 2) + \sqrt{7} + x^3 + 6x^2y + 12xy^2 + 8y^3, 41x - 18y - \sqrt{5} + 8x^3 - 12x^2y + 6xy^2 - y^3 + 3\sqrt{7}(4xy - 4x^2 - y^2 - 2)\}$ at $Z_3 \approx (1.5055, 0.36528)$ with $\mult = 5$ and $o = 4$;
\item[4:] $\{2x_1 + 2x_1^2 + 2x_2 + 2x_2^2 + x_3^2 - 1, 
\mbox{$(x_1 + x_2 - x_3 - 1)^3-x_1^3$}, \\
(2x_1^3 + 5x_2^2 + 10x_3 + 5x_3^2 + 5)^3 - 1000 x_1^5\}$ at
$(0,0,-1)$ with $\mult = 18$ and $o = 7$.
\end{itemize}}

We compare using the following four methods:
%\begin{itemize}
%\item  [A:]
  (A) intrinsic slicing version of \cite{DayZen2005,lvz06};
%\item[B:]
  (B) isosingular deflation \cite{HauWam13} via a maximal rank submatrix;
%\item[C:]
  (C) ``kerneling'' method in \cite{GiuYak13};
%\item[D:]  
  (D) approach of \S~\ref{Sec:Deflation} using an $|\bm i|=1$ step.
%\end{itemize}
We performed these methods without the use of preprocessing and postprocessing
as mentioned in \S~\ref{Sec:Deflation} to directly compare the 
number of nonzero distinct polynomials, variables, and iterations
for each of these four deflation methods.
\vskip -0.05in
{%\tiny
$$
  \begin{array}{|c|c|c|c|c|c|c|c|c|c|c|c|c|}
\hline
 & 
 \multicolumn{3}{|c|}{\hbox{Method A}} & 
 \multicolumn{3}{|c|}{\hbox{Method B}} &
 \multicolumn{3}{|c|}{\hbox{Method C}} & 
 \multicolumn{3}{|c|}{\hbox{Method D}}\\
\cline{2-13}
& 
\hbox{Poly} & \hbox{Var} & \hbox{It} &
\hbox{Poly} & \hbox{Var} & \hbox{It} &
\hbox{Poly} & \hbox{Var} & \hbox{It} &
\hbox{Poly} & \hbox{Var} & \hbox{It} \\
\hline
1 & 16 & 4 & 2 & 22 & 4 & 2 & 22 & 4 & 2 & 16 & 4 & 2 \\
\hline
2 & 24 & 11 & 3 & 11 & 3 & 2 & 12 & 3 & 2 & 12 & 3 & 3 \\
\hline
3 & 32 & 17 & 4 & 6 & 2 & 4 & 6 & 2 & 4 & 6 & 2 & 4 \\
\hline
4 & 96 & 41 & 5 & 54 & 3 & 5 & 54 & 3 & 5 & 22 & 3 & 5 \\
\hline
\end{array}$$}%
For breadth one singular points as in system 3, methods B, C, and D yield
the same deflated system.
Except for methods B and C on the second system, all four methods required the same number of iterations to deflate the root.  
For the first and third systems, our new approach matched
the best of the other methods and resulted in a 
significantly smaller deflated system for~the~last~one.

%\section*{Acknowledgments}
%J.D. Hauenstein is partly supported by NSF grant ACI-1460032 and Sloan Research Fellowship.
%A. Szanto is partly supported by NSF grant CCF-1217557.

%\section*{References}
%\bibliographystyle{abbrv}
%\bibliography{paper}
\def\cprime{$'$} \def\cprime{$'$} \def\cprime{$'$}
\def\sameauthors{------.~}

\end{document}